\documentclass[11pt]{article}%

\usepackage{float} 
\usepackage{fullpage}
\usepackage{amsfonts}
\usepackage{amsmath}
\usepackage{amssymb}
\usepackage{amsbsy}
\usepackage{amsthm}
\usepackage{mathtools}
\usepackage{epsfig}
\usepackage{graphicx}
\usepackage{colordvi}
\usepackage{graphics}
\usepackage{color}
\usepackage{bm}
\usepackage{verbatim} 
\usepackage{multirow}
\usepackage{xfrac}
\usepackage{array}
\usepackage{booktabs}
\usepackage{listings}

\numberwithin{equation}{section}

\newtheorem{theorem}[equation]{Theorem}
\newtheorem{cor}[equation]{Corollary}
\newtheorem{lemma}[equation]{Lemma}
\newtheorem{remark}[equation]{Remark}
\newtheorem{alg}[equation]{Algorithm}
\newtheorem{assumption}[equation]{Assumption}

\newtheorem{definition}[equation]{Definition}

\usepackage{algpseudocode}

\usepackage[dvipsnames]{xcolor}

\newcommand{\be}{\begin{equation}}
	\newcommand{\ee}{\end{equation}}
\newcommand{\bea}{\begin{eqnarray}}
	\newcommand{\eea}{\end{eqnarray}}
\newcommand{\beas}{\begin{eqnarray*}}
	\newcommand{\eeas}{\end{eqnarray*}}

\newcommand{\vertiii}[1]{{\left\vert\kern-0.25ex\left\vert\kern-0.25ex\left\vert #1
		\right\vert\kern-0.25ex\right\vert\kern-0.25ex\right\vert}}

\newcommand{\normiii}[1]{{\left\vert\kern-0.25ex\left\vert\kern-0.25ex\left\vert #1
		\right\vert\kern-0.25ex\right\vert\kern-0.25ex\right\vert}}

\begin{document}
	\title{Anderson acceleration of a Picard solver for the Oldroyd-B model of viscoelastic fluids}
	
	\author{
	Duygu Vargun\thanks{\small
			Computer Science and Mathematics Division, Oak Ridge National Laboratory, Oak Ridge, TN 37831, vargund@ornl.gov; partially supported by NSF grant DMS 2011490 and Department of Energy grant DE-AC05-00OR22725.\\ This research is based upon work supported by the U.S. Department of Energy Office of Science, Office of Advanced Scientific Computing Research, as part of their Applied Mathematics Research Program. The work was performed at the Oak Ridge National Laboratory, which is managed by UT-Battelle, LLC under Contract No. DE-AC05-00OR22725. The United States Government retains and the publisher, by accepting the article for publication, acknowledges that the United States Government retains a non-exclusive, paid-up, irrevocable, world-wide license to publish or reproduce the published form of this manuscript, or allow others to do so, for the United States Government purposes. The Department of Energy will provide public access to these results of federally sponsored research in accordance with the DOE Public Access Plan (http://energy .gov /downloads /doe -public -access -plan).}	
	\and
    	Igor O. Monteiro\thanks{\small Institute of Mathematics, Statistics and Physics, Federal University of Rio Grande, Rio Grande, RS, Brazil, 474, igor.monteiro@furg.br }
        \and
	Leo G. Rebholz\thanks{\small
			School of Mathematical and Statistical Sciences, Clemson University, Clemson, SC, 29364, rebholz@clemson.edu; partially supported by Department of Energy grant DE-SC0025292.}			
	}
	
	\maketitle
	
	\begin{abstract}{
		We study an iterative nonlinear solver for the Oldroyd-B system describing incompressible viscoelastic fluid flow. We establish a range of attributes of the fixed-point-based solver, together with the conditions under which it becomes contractive and examining the smoothness properties of its corresponding fixed-point function. Under these properties, 
        we demonstrate that the solver meets the necessary conditions for recent Anderson acceleration (AA) framework, thereby showing that AA enhances the solver's linear convergence rate. Results from two benchmark tests illustrate how AA improves the solver's ability to converge as the Weissenberg number is increased.
		}
	\end{abstract}

	\section{Introduction}
	
	We consider herein efficient solvers for viscoelastic fluids. These fluids display both fluid-like and solid-like characteristics: they respond to lower strains with an elastic, solid-like behavior, yet when subjected to sufficiently large stresses, they flow as liquids. Due to this behavior, they are classified as non-Newtonian fluids whose viscosity depends on the shear rate or its history. Examples of applications of viscoelastic fluids are blood, mucus, and synovial fluid, which are important for designing medical devices and treatments in biomedical engineering; polymer melts in extrusion processes, injection molding, and blow molding in the plastics industry; and, sealants, adhesives, and coatings in the construction industry. Given their extensive applications, theoretical insights and empirical evidence regarding these fluids are essential in numerous scientific and industrial domains.
	
Numerically simulating the flow of viscoelastic fluids poses a highly challenging nonlinear problem, mainly because the governing equations exhibit hyperbolic characteristics linked to the physical parameter known as the Weissenberg number $\lambda$.  Common solvers (e.g. the Picard type solver studied herein) are contractive for small $\lambda$, but for moderate or large $\lambda$ the solvers are not contractive and tend to fail.  This is not surprising since the physics of viscoelastic flows gets more complex as $\lambda$ increases: when $\lambda\ll 1$ viscous effects dominate but as $\lambda$ increases the (hyperbolic) elastic effects dominate  \cite{Brown_etal93_viscoelastic,Apelian1988_viscoelastic,Baaijens98_viscoelastic,KEUNINGS1986_viscoelastic}. 
Several approaches to improving solvers for flows characterized by higher $\lambda$ and stress responses have been proposed in the literature, including defect correction methods \cite{Ervin06_viscoelastic}, Lagrangian framework for viscoelastic fluids \cite{PETERA2002_viscoelastic}, alternative formulations \cite{PWW17}, and more \cite{Brown_etal93_viscoelastic,Luo98_viscoelastic}.  
	
 We consider herein the steady Oldroyd-B model \cite{Oldroyd58_viscoelastic} for viscoelastic flow, which is given by the system: in $\Omega \subset \mathbb{R}^d$ ($d=$2 or 3) with Lipschitz continuous boundary:
\begin{equation}\label{OldroydB}
\begin{aligned}
\sigma + \lambda   \left((u\cdot\nabla)\sigma- \nabla u \sigma-\sigma\nabla u^T \right)- 2\alpha  \mathbb{D}(u) &= 0\ \text{in}\ \Omega,\\
-\nabla\cdot\sigma-2(1-\alpha)\nabla\cdot \mathbb{D}(u) + \nabla p &= f\ \text{in}\ \Omega,\\
\nabla\cdot u &= 0\ \text{in}\ \Omega,\\
u&=0\ \text{on}\ \partial\Omega,
\end{aligned}
\end{equation}
where $\sigma$ denotes the polymeric stress tensor, $u$ the velocity vector, $p$ the pressure of fluid, $\lambda$ is the Weissenberg number (interpreted here as the product of the relaxation time of the fluid and a characteristic strain rate), and $\mathbb{D}(u)=\left(\nabla u + \nabla u^T \right)/2$ is the deformation tensor. To guarantee the uniqueness of $p$, we impose that its mean over $\Omega$ is zero. The parameter $\alpha$, which must satisfy  $0<\alpha<1$, can be interpreted as the fraction of viscoelastic viscosity, and $f$ denotes the body force. We note that while Oldroyd-B is commonly used, several other viscoelastic fluid models are also used in practice to describe the behavior of complex fluids, with various models capturing different aspects of viscoelastic behavior (e.g.  the Giesekus model includes an anisotropic drag term to capture the nonlinear behavior of polymer solutions \cite{GIESEKUS1982_viscoelastic} and Johnson-Segalman (a generalization of Oldroyd-B) accounts for nonlinear deformation (slip) between the polymer network and the solvent \cite{JOHNSON77_viscoelastic}).  Existence of a weak solution to the problem \eqref{OldroydB}  has been documented by Renardy \cite{Renardy85_viscoelastic}, and moreover, if $f$ is sufficiently regular and small, the problem admits a unique bounded solution $(u,\sigma,p)\in H^3(\Omega)\times H^2(\Omega) \times H^2(\Omega)$. 

We consider herein a solver enhancement called Anderson acceleration (AA) for the purpose of improving convergence of the nonlinear solvers for Oldroyd-B.  As this approach is likely applicable to many of the various models listed above, we consider a simple setting so as to focus on the AA enhancement: a Picard-type iterative solver with conforming finite elements.  The iteration is defined by:
given $u_{k}$, determine $u_{k+1},\sigma_{k+1},p_{k+1}$ satisfying
\begin{equation}\label{OldroydBPicard}
\begin{aligned}
\sigma_{k+1} + \lambda \left((u_{k}\cdot\nabla)\sigma_{k+1}- \nabla u_{k} \sigma_{k+1} - \sigma_{k+1}\nabla u_{k}^T \right)- 2\alpha  \mathbb{D}(u_{k+1}) &= 0\ \text{in}\ \Omega,\\
-\nabla\cdot\sigma_{k+1}-2(1-\alpha)\nabla\cdot \mathbb{D}(u_{k+1}) + \nabla p_{k+1} &= f\ \text{in}\ \Omega,\\
\nabla\cdot u_{k+1} &= 0\ \text{in}\ \Omega,\\
u_{k+1}&=0\ \text{on}\ \partial\Omega.
\end{aligned}
\end{equation}
The iteration \eqref{OldroydBPicard} may be taken into account as a fixed-point iteration of the form $(u_{k+1},\sigma_{k+1},p_{k+1})=G(u_{k},\sigma_{k},p_{k})$, and we will show in section 3 that it has a contraction ratio that depends directly on $\lambda$. If $\lambda$ is small, then the iteration is contractive and linearly convergent.  However, for larger $\lambda$, the iteration is generally not contractive and convergence is not expected.  

Hence, we employ AA to enhance the convergence behavior of solver \eqref{OldroydBPicard}, as AA is recognized as a potent extrapolation technique that accelerates fixed-point iterations. This includes accelerating iterations that are already converging but also enabling convergence of iterations that would otherwise not converge such as in solvers for Navier-Stokes as the Reynolds number increases \cite{PRX19} and Boussinesq flows as the Rayleigh number increases \cite{PRX21}.  AA was originally introduced by D.G. Anderson in 1965 \cite{Anderson65}, and although it had modest use until 2011, it was after the Walker and Ni paper \cite{WaNi11} that AA caught fire across nearly all of computational science as a simple way to improve most types of nonlinear solvers.  AA has recently been used to in a wide range of problems including various types of Newtonian flow  \cite{LWWY12,PRX19,PRX21, LVX21} and non-Newtonian flow problems \cite{PRV23_AA}, electromagnetic nonlinear and \cite{FAC19_AA} geometry optimization problems \cite{PDZGQL2018_AA}
radiation diffusion and nuclear physics \cite{AJW17,TKSHCP15}, geophysics \cite{Yang2021_AA}
molecular interaction \cite{SM11},  and many others e.g. \cite{WaNi11,K18,LW16,LWWY12,FZB20,WSB21_AA}. Motivated by the proven success of AA in a range of nonlinear problems from above references, we now study it with viscoelastic flow. In this work, we demonstrate both theoretically and numerically that Picard enhanced with AA retains the simplicity of the standard Picard iteration while significantly improving its efficiency and robustness, particularly for larger values of $\lambda$ in \eqref{OldroydBPicard}. Although our focus here is on the steady-state case, it is straightforward to adapt our methods and analysis to AA for solving the nonlinear systems that arise at each time step in the time-dependent problem.

The structure of the paper is as follows. In the next section, we provide the essential notation and preliminary information necessary for our subsequent analysis. In Section 3, we provide the well-posedness and convergence analyses of the discretized fixed-point iteration \eqref{OldroydBPicard} of Oldroyd-B. Section 4 introduces AA for \eqref{OldroydBPicard}, and proves that the iteration's fixed-point function has sufficient smoothness properties so that the recently developed convergence theory for AA can be invoked. In Section 5, we support the theory we discussed in previous sections with numerical tests, and show how AA improves and enables convergence. Lastly, we draw conclusions in Section 6.  The FENICS code we use for AA is given in the Appendix.

\section{Mathematical preliminaries}

The function spaces for the velocity, pressure, and stress for \eqref{OldroydB} are given by (respectively):
\begin{align*}
X&\coloneqq \left(H^1_0(\Omega)\right)^d=\left\{ v\in H^1(\Omega): v=0\ \text{on}\ \partial\Omega \right\},
\\
Q&\coloneqq L^2_0(\Omega)=\left\{ q\in L^2(\Omega): \int_{\Omega} q\ d\Omega = 0 \right\},\\
\Sigma &\coloneqq \left( L^2(\Omega)^{d\times d}\right) \bigcap \left\{ \tau=(\tau_{ij}):\tau_{ij}=\tau_{ji}, u\cdot\nabla\tau\in(L^2(\Omega))^{d\times d} \right\}. 
\end{align*}
where $\Omega\subset\mathbb{R}^d\ (d=2,\ 3)$ domain of the problem which is polygonal for $d=2$ or polyhedral for $d=3$ (or $\partial\Omega\in C^{0,1}$). Additionally, the weakly divergence free subspace of $X$ is given by
\begin{align*}
V&=\left\{ v\in X: \int_{\Omega}q\nabla\cdot v\ d\Omega=0\ \forall q\in L^2_0(\Omega) \right\}.
\end{align*}

The notation $(\cdot,\cdot)$ will be used to denote the $L^2$ inner product with its norm denoted by $\|\cdot\|$. For notational simplicity, we will also use this notation
for the duality pairing between $H^{-1}(\Omega)$ and $X$ with its norm denoted by $\|\cdot\|_{-1}$. Note that, $\|\cdot\|_{\infty}$ denotes then $L^{\infty}(\Omega)-$norm.

In the space $X$,  the Poincar\'e inequality holds \cite{Laytonbook}: there exists a constant $C_P>0$ depending only on $\Omega$ such that for any $\phi\in X$,
\[
\| \phi \| \le C_P \| \nabla \phi \|.
\]
We also use the Korn type inequality \cite{BS08}:
\begin{equation}
\| \nabla \phi \| \le C_K \| \mathbb{D}(\phi) \|,\ \forall \phi\in X. \label{korn}
\end{equation}

\subsection{Finite element discretization of Oldroyd-B model}
We apply finite element methods to discretize equation~\eqref{OldroydB} on a regular, conforming mesh $\tau_h(\Omega)$ with a maximum element size of $h$. The pair of discrete velocity-pressure spaces $(X_h,Q_h)\subset (X,Q)$ must satisfy 
\begin{align}\label{discinfsup}
\inf_{q\in Q_h} \sup_{v\in X_h}\frac{(\nabla \cdot v, q)}{\|q\| \|\nabla v\|}\geq \beta >0,
\end{align}
which is called the inf-sup (or LBB) condition. In \eqref{discinfsup}, $\beta$ is a constant independent of $h$.

We denote $P_m(\tau_h)$ as the set of degree $m$ piecewise polynomial functions that are continuously connected across the entire mesh $\tau_h$. In contrast, $P_m^{disc}(\tau_h)$ represents the set of degree $m$ piecewise polynomial functions that are allowed to have discontinuities between across elements.

One of the common choices in this setting that satisfy \eqref{discinfsup}, is Taylor-Hood elements, $(X_h,Q_h)=(P_2(\tau_h)\cap X, P_1(\tau_h)\cap Q))$ along with the discrete stress space $\Sigma_h=P_1^{disc}(\tau_h)\cap \Sigma$.

We can now write the FEM scheme for \eqref{OldroydB}: Find $(u,p,\sigma)\in (X_h,Q_h,\Sigma_h)$ such that
\begin{equation}\label{weakOldroydB}
\begin{aligned}
(\sigma,\tau) + \lambda  \left( b(u,\sigma,\tau) - (\nabla u \sigma,\tau)-(\sigma\nabla u^T,\tau) \right)- 2\alpha  (\mathbb{D}(u),\tau) &= 0\ \forall \tau\in\Sigma_h,\\
(\sigma,\mathbb{D}(v))+2(1-\alpha)(\mathbb{D}(u),\mathbb{D}(v)) - (p,\nabla\cdot v) &= (f,v)\ \forall v\in X_h, \\
(q,\nabla\cdot u)&=0\ \forall q\in Q_h,
\end{aligned}
\end{equation}
where the trilinear form $b:X_h \times \Sigma_h \times \Sigma_h \rightarrow \mathbb{R}$ is defined by
\begin{align}\label{divform}
b(u,\sigma,\tau)=((u\cdot\nabla)\sigma,\tau)+\frac{1}{2} ((\nabla\cdot u)\sigma,\tau).    
\end{align}
We note $b$ is equivalent to the usual convective form $((u\cdot\nabla)\sigma,\tau)$ when $\| \nabla \cdot u\|=0$. The motivation behind using this divergence form is that \eqref{divform} will vanish if $\sigma=\tau$, even if the first argument is not divergence-free; this is the typical case for finite element method discretizations.  Existence and uniqueness of the weak solution to \eqref{weakOldroydB} is discussed in \cite{Ervin05_viscoelastic,Renardy85_viscoelastic}.

For notational simplicity, we can consider the equivalent problem set in discretely divergence free functions
$$V_h:=\{ v\in X_h: (\nabla\cdot v,q)=0\ \forall q\in Q_h \},$$
such that
\begin{equation}\label{weakdivfreeOldroydB}
\begin{aligned}
(\sigma,\tau) + \lambda  \left( b(u,\sigma,\tau) - (\nabla u \sigma,\tau)-(\sigma\nabla u^T,\tau) \right)- 2\alpha  (\mathbb{D}(u),\tau) &= 0\ \forall \tau\in\Sigma_h,\\
(\sigma,\mathbb{D}(v))+2(1-\alpha)(\mathbb{D}(u),\mathbb{D}(v)) &= (f,v)\ \forall v\in V_h.   \end{aligned}
\end{equation}
For $v\in V_h$ and $s \in \Sigma_h$, the notation $\| (\cdot,\cdot) \|_*$ is used for the natural problem norm defined by
\[
\| (v,s) \|_* := \left( \| \mathbb{D}(v) \|^2 + \| s \|^2 \right)^{1/2}.
\]

We assume the mesh is sufficiently regular for the inverse inequality to hold  (see Theorem 4.5.11 and Remark 4.5.20 in \cite{BS08}): there exists a constant $ C_I >0$ such that
\begin{align}\label{invineq}
\|\nabla \sigma\|\leq C_I h^{-1} \|\sigma\|,\ \forall \sigma\in\Sigma_h.
\end{align}

\section{Analysis of fixed-point iteration of Oldroyd-B model}

In this section, we analyze the Picard-type iteration of Oldroyd-B model in $V_h-$formulation, which is defined as follows: Find $(u_{k+1},\sigma_{k+1})\in (V_h,\Sigma_h)$ satisfying all $(v,\tau)\in (V_h,\Sigma_h)$ such that
\begin{equation}\label{divfreeOldroydBPicardScheme}
\begin{aligned}
(\sigma_{k+1}, \tau) 
+ \lambda \left(
b(u_k,\sigma_{k+1},\tau)
- (\nabla u_{k} \sigma_{k+1},\tau) 
- (\sigma_{k+1}\nabla u_{k}^T,\tau) 
\right)
- 2\alpha (\mathbb{D}(u_{k+1}),\tau) &= 0,\\
(\sigma_{k+1},\mathbb{D}(v))
+2(1-\alpha)(\mathbb{D}(u_{k+1}),\mathbb{D}(v)) &= (f,v).  \end{aligned}
\end{equation}

\begin{lemma}\label{lem:stability}(Stability) Assume the true velocity solution $u$ to \eqref{OldroydB}
 and previous iterate $u_k$ to \eqref{divfreeOldroydBPicardScheme} is sufficiently close to $u$ so that
 \[
 \| \nabla (u_{k}- u )\|_{\infty}+\| \nabla u\|_{\infty} < \frac{1}{2\lambda}.
\]
Then the iteration \eqref{divfreeOldroydBPicardScheme} at step $k+1$ is bounded:
    \begin{equation}
\begin{aligned}
\left(1-
  2\lambda \| \nabla (u_{k}- u )\|_{\infty}
-
  2\lambda \| \nabla u\|_{\infty}
\right)\|\sigma_{k+1}\|^2
+
2\alpha(1-2\alpha)\|\mathbb{D}(u_{k+1})\|^2
 &
 \leq
\frac{\alpha^{-1} C_K^2}{8} \|f\|^2_{-1}. 
\end{aligned}
\end{equation}
\end{lemma}

\begin{proof}
    Assume that the solutions $(u_{k+1},\sigma_{k+1})$ exist and set  $(v,\tau)=(u_{k+1},\sigma_{k+1})$ in \eqref{divfreeOldroydBPicardScheme} to get
\begin{equation}
\begin{aligned}
\|\sigma_{k+1}\|^2
- \lambda \left(
(\nabla u_{k} \sigma_{k+1},\sigma_{k+1}) 
+(\sigma_{k+1}\nabla u_{k}^T,\sigma_{k+1}) 
\right)
- 2\alpha (\mathbb{D}(u_{k+1}),\sigma_{k+1}) &= 0,\\
(\sigma_{k+1},\mathbb{D}(u_{k+1}))
+2(1-\alpha)\|\mathbb{D}(u_{k+1})\|^2 &= (f,u_{k+1}),  
\end{aligned}
\end{equation}
noting $b(u_k,\sigma_{k+1},\sigma_{k+1})$ vanishes. Then, multiplying the second equation by $2\alpha$ and adding it to the first one provides
\begin{equation} \label{3p5}
\begin{aligned}
\|\sigma_{k+1}\|^2
+
4\alpha(1-\alpha)\|\mathbb{D}(u_{k+1})\|^2
 &=(f,u_{k+1})
 + \lambda 
(\nabla u_{k} \sigma_{k+1},\sigma_{k+1}) 
 + \lambda(\sigma_{k+1}\nabla u_{k}^T,\sigma_{k+1}).
\end{aligned}
\end{equation}
The first term on the right hand side can be bounded as
\begin{align*}
 (f,u_{k+1})\leq \frac{\alpha^{-1} C_K^2}{8} \|f\|^2_{-1} +   2\alpha \| \mathbb{D} u_{k+1}\|^2
\end{align*}
thanks to Cauchy-Schwarz and Young's inequalities with the dual norm on $V_h$ and \eqref{korn}.
The last two terms on the right hand side of \eqref{3p5} can be bounded by applying H\"older's inequality, then adding and subtracting $u$ to obtain
\begin{align*}
\lambda 
(\nabla u_{k} \sigma_{k+1},\sigma_{k+1}) 
 + \lambda(\sigma_{k+1}\nabla u_{k}^T,\sigma_{k+1})
& \leq
 2 \| \nabla u_k \|_{\infty} \| \sigma_k \|^2\\
 & \le
  2\lambda 
\left(
\| \nabla (u_{k}- u )\|_{\infty}
+
\| \nabla u\|_{\infty}
\right)
\|\sigma_{k+1}\|^2.
 \end{align*}
Then, by combining these bounds we get
\begin{equation}
\begin{aligned}
\left(1-
  2\lambda \| \nabla (u_{k}- u )\|_{\infty}
-
  2\lambda \| \nabla u\|_{\infty}
\right)\|\sigma_{k+1}\|^2
+
2\alpha(1-2\alpha)\|\mathbb{D}(u_{k+1})\|^2
 &
 \leq
\frac{\alpha^{-1} C_K^2}{8} \|f\|^2_{-1}. 
\end{aligned}
\end{equation}
Using the assumptions on $u$ and $u_k$ finishes the proof.
\end{proof}

\begin{lemma}\label{lem:uniqueness}(Uniqueness) 
Under the same assumptions as Lemma \ref{lem:stability}, the Step k+1 solution of \eqref{divfreeOldroydBPicardScheme} is unique.  Moreover, since the system is finite dimensional and linear, solution existence and therefore well-posedness follows immediately.

\end{lemma}
\begin{proof}
Assume that $(u_{k+1}^1,\sigma_{k+1}^1)$ and $(u_{k+1}^2,\sigma_{k+1}^2)$ are two solutions of \eqref{divfreeOldroydBPicardScheme}, then let $\phi^u=u^1-u^2$ and $\phi^{\sigma}=\sigma^1-\sigma^2$. Then, for given $u_k,\sigma_k$ and $f$, we have that

\begin{align*}
(\phi^{\sigma}_{k+1}, \tau) 
+ \lambda \left(
b(u_k,\phi^{\sigma}_{k+1},\tau)
- (\nabla u_{k} \phi^{\sigma}_{k+1},\tau) 
- (\phi^{\sigma}_{k+1}\nabla (u_{k})^T,\tau) 
\right)
- 2\alpha (\mathbb{D}(\phi^{u}_{k+1}),\tau) &= 0,\\
(\phi^{\sigma}_{k+1},\mathbb{D}(v))
+2(1-\alpha)(\mathbb{D}(\phi^{u}_{k+1}),\mathbb{D}(v)) &= 0. 
\end{align*}
Setting $(v,\tau)=( \phi^{u}_{k+1}, \phi^{\sigma}_{k+1})$ provides
\begin{align*}
\|\phi^{\sigma}_{k+1}\|^2 
+4\alpha(1-\alpha) \|\mathbb{D}(\phi^{u}_{k+1})\|^2
 = 
 \lambda \left(
(\nabla u_{k} \phi^{\sigma}_{k+1},\phi^{\sigma}_{k+1}) 
+ (\phi^{\sigma}_{k+1}\nabla (u_{k})^T,\phi^{\sigma}_{k+1}) 
\right),
\end{align*}
by adding $2\alpha$ times of second equation into the first one while nonlinear term vanishes. By applying H\"older's inequality to the right hand side terms where we add and subtract $u$ to the velocity terms, we get
\begin{align}\label{eqn:uniquness}
\left(1-2\lambda\|\nabla (u - u_{k})\|_{\infty} -2\lambda\|\nabla u\|_{\infty} \right)
\|\phi^{\sigma}_{k+1}\|^2 
+4\alpha(1-\alpha) \|\mathbb{D}(\phi^{u}_{k+1})\|^2
\leq 0.
\end{align}
Since $0<\alpha<1$ and from Lemma \ref{lem:stability}, we know $\left(1-2\lambda\|\nabla( u-u_{k})\|_{\infty} -2\lambda\|\nabla u\|_{\infty} \right) \geq 0$, hence the inequality \ref{eqn:uniquness} can only be true for $\phi^{u}_{k+1}=\phi^{\sigma}_{k+1}=0$. This implies the solution of \eqref{divfreeOldroydBPicardScheme} is unique.
\end{proof}

Next, we discuss the convergence behavior of \eqref{divfreeOldroydBPicardScheme}.
\begin{theorem}
Suppose that the solution of \eqref{weakdivfreeOldroydB} satisfies $(u,\sigma)\in H^3(\Omega)\times H^2(\Omega)$ and let  $(u_{k+1},\sigma_{k+1})\in V_h \times \Sigma_h$ denote the solution of \eqref{divfreeOldroydBPicardScheme}. Then under the same assumptions as Lemma \ref{lem:stability}, 
\begin{align}\label{conv}
\|(e^{u}_{k+1},e^{\sigma}_{k+1})\|_{*}
\leq
\sqrt{\frac{\lambda  \left( C_{P-K}^2 \lambda \|\nabla \sigma\|_{\infty}^2 +C_K^2 \|\sigma\|_{\infty}^2 \right)}{2S}}	\|(e^{u}_{k},e^{\sigma}_{k})\|_{*},
\end{align}
where $S\coloneqq\min \{ \left(1-  2 \lambda\|\nabla u\|_{\infty} - 2 \lambda \|\nabla (u-u_k)\|_{\infty}\right),4\alpha(1-\alpha) \}$.
\end{theorem}

\begin{remark}
For Theorem \ref{conv} to imply the entire iteration is contractive, we need the assumption of $S>0$ at Step k to imply $S>0$ at Step k+1, i.e. that $u_{k+1}$ is closer to $u$ than $u_k$ is in $L^{\infty}$.  From \eqref{thmpf2} in the proof of Theorem \ref{conv}, we observe that 
\[
\| \mathbb{D}(e_{k+1}^u) \| \le \left( \frac{\lambda}{8\alpha(1-\alpha)} \left( C_{P-K}^2 \lambda \|\nabla \sigma\|_{\infty}^2 +C_K^2 \lambda \|\sigma\|_{\infty}^2 \right) \right)^{1/2} \| \mathbb{D}(e^{u}_{k})\|.
\]
Hence smallness assumptions on the problem data (which imply smallness of the true solution) will imply that the generated sequence of velocities is contractive in the $H^1$ norm.  To extend this to the $L^{\infty}$ norm, it seems necessary to employ an inverse inequality which will make the smallness assumption smaller in order to guarantee contractiveness in this norm.  We note that under this smallness assumption, the above lemmas imply we have well-posedness of all steps of the iteration as well as a uniform bound on solutions.
\end{remark}

\begin{proof}
Subtracting \eqref{divfreeOldroydBPicardScheme} from \eqref{weakdivfreeOldroydB}, then denoting $e^{\sigma}_{k}=\sigma-\sigma_{k}$ and $e^{u}_{k}=u-u{_k}$ provides
\begin{align*}
(e^{\sigma}_{k+1},\tau)
+ 
\lambda (
b(e^{u}_{k},\sigma,\tau)
+
b(u_{k},e^{\sigma}_{k+1},\tau)
- (\nabla e^u_{k} \sigma ,\tau)
\\
- (\nabla u_{k} e^{\sigma}_{k+1} ,\tau)
-( \sigma \nabla (e^u_{k})^T,\tau)
-( e^{\sigma}_{k+1}\nabla u_{k}^T ,\tau)
)
- 2\alpha  (\mathbb{D}(e^u_{k+1}),\tau) 
&= 0,\\
( e^{\sigma}_{k+1}, \mathbb{D} (v))
+2(1-\alpha)(\mathbb{D}(e^u_{k+1}), \mathbb{D}( v))
&= 0.
\end{align*}
Set $\tau=e^{\sigma}_{k+1}$ in the first equation, which vanishes the nonlinear term,
and  $v=e^{u}_{k+1}$ in the second equation to get
\begin{align*}
\|e^{\sigma}_{k+1}\|^2
+ 
\lambda b( e^{u}_{k},\sigma ,e^{\sigma}_{k+1})
- \lambda(\nabla e^u_{k} \sigma ,e^{\sigma}_{k+1})
- \lambda(\nabla u_{k} e^{\sigma}_{k+1},e^{\sigma}_{k+1})
\\ 
- \lambda(\sigma \nabla (e^u_{k})^T,e^{\sigma}_{k+1})
-\lambda( e^{\sigma}_{k+1}\nabla u_{k}^T ,e^{\sigma}_{k+1})
- 2\alpha  (\mathbb{D}(e^u_{k+1}),e^{\sigma}_{k+1})) 
&= 0,\\
(e^{\sigma}_{k+1},\mathbb{D} (e^{u}_{k+1}))
+2(1-\alpha) \|\mathbb{D}(e^u_{k+1})\|^2
&= 0.
\end{align*}
Next, we multiply second equation by $2\alpha$ and add it to the first one, which provides
\begin{align*}
\|e^{\sigma}_{k+1}\|^2x
+4\alpha(1-\alpha) \|\mathbb{D}(e^u_{k+1})\|^2
=&
-\lambda b( e^{u}_{k},\sigma ,e^{\sigma}_{k+1})
+ \lambda(\nabla e^u_{k} \sigma ,e^{\sigma}_{k+1})
+ \lambda(\nabla u_{k} e^{\sigma}_{k+1},e^{\sigma}_{k+1}) 
\\&
+ \lambda(\sigma \nabla (e^u_{k})^T,e^{\sigma}_{k+1})
+\lambda( e^{\sigma}_{k+1}\nabla u_{k}^T ,e^{\sigma}_{k+1}).
\end{align*}
By H\"older's inequalities, we get  
\begin{multline*}
\|e^{\sigma}_{k+1}\|^2
+4\alpha(1-\alpha) \|\mathbb{D}(e^u_{k+1})\|^2
\\
\leq
\lambda \|\nabla \sigma\|_{\infty} \|e^{u}_{k}\|  \|e^{\sigma}_{k+1}\|
+\lambda \|\sigma\|_{\infty} \|\nabla e^{u}_{k}\| \|e^{\sigma}_{k+1}\|
+ 2 \lambda \left(\|\nabla (u-u_k)\|_{\infty} + \|\nabla u\|_{\infty}\right) \|e^{\sigma}_{k+1}\|^2.
\end{multline*}
Next, we reduce the last equation and use \eqref{korn}, Poinc\'are and Young inequalities to obtain
\begin{multline}
 \left(1- 2 \lambda\|\nabla  u\|_{\infty} - 2 \lambda \|\nabla (u_k - u) \|_{\infty}\right)\|e^{\sigma}_{k+1}\|^2
+4\alpha(1-\alpha) \|\mathbb{D}(e^u_{k+1})\|^2
\\
\leq
\frac{\lambda}{2} \left( C_{P-K}^2 \lambda \|\nabla \sigma\|_{\infty}^2 +C_K^2 \|\sigma\|_{\infty}^2 \right) \| \mathbb{D}(e^{u}_{k})\|^2, \label{thmpf2}
\end{multline}
where $C_{P-K}$ is the constant absorbing constants coming from Poinc\'are and \eqref{korn}.
After dividing both sides by $S$ we get
\begin{align*}
\|(e^{u}_{k+1},e^{\sigma}_{k+1})\|^2_{*}
=
\|e^{\sigma}_{k+1}\|^2
+\|\mathbb{D}(e^u_{k+1})\|^2
&\leq
\frac{\lambda  \left( C_{P-K}^2 \lambda \|\nabla \sigma\|_{\infty}^2 +C_K^2 \|\sigma\|_{\infty}^2 \right)}{2S}\|\mathbb{D}( e^{u}_{k})\|^2
\\	&\leq
\frac{\lambda  \left( C_{P-K}^2 \lambda \|\nabla \sigma\|_{\infty}^2 +C_K^2 \|\sigma\|_{\infty}^2 \right)}{2S}	\|(e^{u}_{k},e^{\sigma}_{k})\|^2_{*}.
\end{align*}
Finally, we take the square root of both sides and get \eqref{conv}.
\end{proof}

\section{Acceleration of the Picard iteration for Oldroyd-B model}

In this section, we analyze the fixed-point iteration for Oldroyd-B model \eqref{divfreeOldroydBPicardScheme} in terms of properties required to apply the convergence and acceleration theory of \cite{PR21} for AA applied to this iteration. We begin by introducing AA and outlining the sufficient conditions that fixed-point schemes must meet to achieve accelerated convergence.

\subsection{Anderson acceleration background}\label{aasection}

We now describe the AA method, an extrapolation approach designed to enhance the convergence of fixed-point iterations governed by operator $g:Y\rightarrow Y$.
\begin{alg} \label{alg:anderson}
(Anderson acceleration with depth $m$ and damping  $\beta_k$)
\begin{enumerate}
\item Initialization: Select an initial vector $x_0\in Y.$
\item First iteration: Compute the residual $w_1\in Y $ such that $w_1 = g(x_0) - x_0$ and update $x_1 = x_0 + w_1$. 
\item Subsequent iterations (for $k=2,3,\ldots$)

Set $m_k = \min\{ k-1, m\}.$
\begin{itemize}   
\item [a.] Compute the residual $w_{k} = g(x_{k-1})-x_{k-1}$. \\
\item [b.] Find optimal the Anderson coefficients $\{ \alpha_{j}^{k}\}_{k-m_k}^{k-1}$ by the following minimization problem 
\begin{align}\label{eqn:opt-v0}
\textstyle \min\limits_{ \sum\limits_{j = k-m_k}^{k-1}\alpha_j^{k}=1} 
\left\|  \sum\limits_{j = k-m_k}^{k-1}  \alpha_j^{k}  w_{j} \right\|_Y.
\end{align}
\item [c.] Update the iterate such that given the damping parameter $0 < \beta_k \le 1$, define
\begin{align}\label{eqn:update-v0}
\textstyle
x_{k} 
=\sum\limits_{j= k-m_k}^{k-1} \alpha_j^{k} x_{j-1}
+ \beta_k  \sum\limits_{j=k-m_k}^{k-1}\alpha_j^k w_{j}.
\end{align}
\end{itemize}
\end{enumerate}
\end{alg}









When $m=0$, the AA method reduces to the basic fixed-point iteration without any acceleration. The recent theory proposed in \cite{PR21, PRX19} shows that acceleration in convergence rate is controlled by gain factor obtained from the optimization step which is defined by
\begin{align*}
\theta_k=\frac{\left\Vert \sum\limits_{j= k-m_k}^{k-1} \alpha_j^{k} x_{j} \right\Vert_Y}{\|w_k\|_Y}.
\end{align*}
The subsequent two assumptions from \cite{PR21} establish sufficient conditions on the fixed-point operator $g$ for the convergence and acceleration results outlined in that work to hold true.

\begin{assumption}\label{assume:g}
Suppose \(g \in C^1(Y)\) admits a fixed point \(x^\ast \in Y\), and there are positive 
constants \(C_L\) and \(\hat{C}_L\) such that:
\begin{enumerate}
\item \(\|g'(x)\|_{Y} \le C_0\) \(\forall x \in Y\),
\item \(\|g'(x) - g'(y)\|_{Y} \le C_1\,\|x - y\|_{Y}\)  \(\forall x, y \in Y\).
\end{enumerate}
\end{assumption}

\begin{assumption}\label{assume:fg}
Assume there is a constant \(\sigma > 0\) satisfying
\[
\|w_{k+1} - w_k\|_Y \;\ge\; \sigma\,\|x_k - x_{k-1}\|_Y,\ k \ge 1.
\]
\end{assumption}

Assumption \ref{assume:g} provides smoothness conditions on the fixed-point operator. In the following analysis, we will show that these conditions hold for the Picard fixed-point operator. Meanwhile, Assumption \ref{assume:fg} is more difficult to verify in this setting. It does hold globally if $g$ is contractive, which we have demonstrated under a suitable smallness condition. It also holds locally if the Jacobian of $g$ remains nondegenerate in a neighborhood of the solution—though we cannot confirm that for the present problem. As noted in \cite{PR21}, this requirement can be monitored during the algorithm, with a restart of Anderson Acceleration triggered if the condition is violated.

Under Assumptions \ref{assume:g} and \ref{assume:fg}, the following result summarized from \cite{PR21} produces a one-step bound on the residual $\|w_{k+1}\|$ in terms of the previous residuals.
\begin{theorem}  \label{thm:genm}
Suppose Assumptions \ref{assume:g} and \ref{assume:fg} hold. For each $j$,
\begin{align*}
		F_j&:=\left( (w_{j}-w_{j-1})(w_{j-1}-w_{j-2})\ \dotsc\ (w_{j-m_j+1}-w_{j-m_j})\right).
	\end{align*}
Further assume that for every column $f_{j,i}$ of $F_j$, the angle between $f_{j,i}$ and the subspace spanned by the preceding columns is uniformly bounded away from zero such that 
$$|\sin(f_{j,i},\text{span }\{f_{j,1}, \ldots, f_{j,i-1}\})| \ge c_s >0,\ j = k-m_k, \ldots, k-1.$$
Then, for Algorithm \ref{alg:anderson} with the depth $m$, the residual  $w_{k+1} = g(x_k)-x_k$ satisfies
\begin{align}\label{eqn:genm}
\|w_{k+1}\| & \le \|w_k\| \Bigg(
\theta_k ((1-\beta_{k}) + C_0 \beta_{k})
+ \frac{C  C_1\sqrt{1-\theta_k^2}}{2}\bigg(
\|w_{k}\|h(\theta_{k})
\nonumber \\ &
+ 2  \sum_{n = k-{m_{k}}+1}^{k-1} (k-n)\|w_n\|h(\theta_n) 
+ m_{k}\|w_{k-m_{k}}\|h(\theta_{k-m_{k}})
\bigg) \Bigg),
\end{align}
where  each $h(\theta_j) \le C \sqrt{1 - \theta_j^2} + \beta_{j}\theta_j$,
and $C$ depends on $c_s$ and on the corresponding upper bound for the direction cosines. 
\end{theorem}


In this bound, $\theta_k$ represents the gain from the optimization step, dictating the relative scaling of lower-order versus higher-order terms. The lower-order terms are multiplied by $\theta_k$, while the higher-order terms are multiplied by $\sqrt{1-\theta_k^2}$. Although the assumption regarding direction sines between columns may fail in practice for larger $m$ the filtering technique introduced in \cite{PR23} can be used to address this issue as the method proceeds.

In the remainder of this section, we define the fixed-point operator linked to iteration \eqref{divfreeOldroydBPicardScheme} and prove it satisfies Assumption  \ref{assume:g} on any mesh. Assumption \ref{assume:fg}, for instance, holds if the iteration is contractive, which—by Lemma \ref{lem:stability}—occurs for sufficiently small data and an appropriately chosen initial \ref{thm:genm} can be applied to the Anderson-accelerated Picard solver \eqref{divfreeOldroydBPicardScheme}, implying its linear convergence rate is modulated by the gain $\theta_k$ derived from the optimization.

\subsection{Fixed-point operator $G$ and associated properties}\label{Gsec}
Suppose that $f\in H^{-1} (\Omega)$ and $u\in V_h$ are given. We aim to find $(\tilde u,\tilde \sigma) \in (V_h, \Sigma_h)$ such that
\begin{equation}\label{discretedivfreeOldroydB}
\begin{aligned}
\left(\tilde\sigma,\tau\right) 
+
\lambda  \left( 
b(u,\tilde\sigma,\tau) 
- \left(\nabla u \tilde\sigma,\tau\right)
-\left(\tilde\sigma\nabla u^T,\tau\right) \right)
-
2\alpha  \left(\mathbb{D}(\tilde u),\tau\right) 
&=
0,\ \forall \tau\in\Sigma_h,
\\
\left(\tilde\sigma,\mathbb{D}(v)\right)
+
2(1-\alpha)\left(\mathbb{D}(\tilde u),\mathbb{D}(v)\right) 
&=
\left(f,v\right),\ \forall v\in V_h.   \end{aligned}
\end{equation}

\begin{lemma}\label{lemma:welldefiniteG}
Suppose that $f\in H^{-1}(\Omega)$ and $(u,\sigma)\in(V_h,\Sigma_h)$ are given. Then, the solution of the system \eqref{discretedivfreeOldroydB} satisfies
\begin{align}\label{eqn:welldefiniteG}
    \|(\tilde u,\tilde\sigma)\|_{*} \leq  \frac{C_K}{\sqrt{4C_0 (1-\alpha)}}\|f\|_{-1},
\end{align}
under assumptions $ C_0 = \min \left\{   \frac{1- 2\lambda\|\nabla u\|_{\infty}}{2\alpha},1-\alpha \right\}$, and $1- 2\lambda\|\nabla u\|_{\infty} \geq 0$, which implies  well-posed the solution.
\end{lemma}

\begin{proof}
Choose $\tau=\tilde\sigma\in \Sigma_h$ which vanishes $b(u,\tilde\sigma,\tilde\sigma) $and $v=\tilde u\in V_h$, then add two equations by multiplying second one by $2\alpha$ to get
\begin{align*}
\|\tilde\sigma\|^2
+
4\alpha(1-\alpha)\|\mathbb{D}(\tilde u)\|^2
&=
2\alpha \left(f,\tilde u\right)
+
\lambda  \left(  
\left(\nabla u \tilde\sigma,\tilde\sigma\right)
+
\left(\tilde\sigma\nabla u^T,\tilde\sigma\right) 
\right)
\\
&\leq
2\alpha\|f\|_{-1} \|\nabla\tilde u\| 
+\lambda
2 \|\nabla u\|_{\infty} \|\tilde\sigma\|^2,  
\end{align*}
thanks to the dual norm on $V_h$ and  H\"older's inequality. Then, dividing both side by $2\alpha$ and applying \eqref{korn} and Young's inequality leads to the following bound:
\begin{align*}
\frac{1- 2\lambda\|\nabla u\|_{\infty}}{2\alpha}\|\tilde\sigma\|^2
+
(1-\alpha)\|\mathbb{D}(\tilde u)\|^2
&\leq
\frac{C_K^2}{4(1-\alpha)} \|f\|^2_{-1}.  
\end{align*}
Setting $ C_0 = \min \left\{   \frac{1- 2\lambda\|\nabla u\|_{\infty}}{2\alpha},1-\alpha \right\}$, dividing both side by $C_0$ and taking square root of both side conclude the proof with \eqref{eqn:welldefiniteG}.
\end{proof}
\begin{definition}\label{defn:G}
Let $G:(V_h,\Sigma_h)\rightarrow (V_h,\Sigma_h)$ denote the the solution operator associated with  \eqref{discretedivfreeOldroydB} such that
\[
(\tilde u, \tilde\sigma) =\left(G_{1}(u,\sigma),G_2(u,\sigma)\right) =G(u,\sigma).
\]
\end{definition}
From Lemma \ref{lemma:welldefiniteG}, the scheme \eqref{discretedivfreeOldroydB} is well-posed, which ensures $G$ is well-defined. Therefore, the iterative scheme \eqref{divfreeOldroydBPicardScheme} takes form with 
\[
(u_{k+1}, \sigma_{k+1})=\left(G_{1}(u_{k},\sigma_{k}),G_2(u_{k},\sigma_{k})\right)=G(u_{k},\sigma_{k}).
\]

The remainder of this subsection examines the operator $G$ in terms of continuity, differentiability, and related regularity properties. We begin by establishing that $G$ is Lipschitz continuous.
\begin{lemma}\label{lemma:lipschitz}
Provided $1- 2\lambda \|\nabla u\|_{\infty} >0$, for any  $(v,\sigma),(w,\rho)\in (V_h,\Sigma_h)$ satisfies \eqref{discretedivfreeOldroydB}
\begin{align}\label{eqn:lipschitz}
\|G(v,\sigma)-G(w,\rho)\|_{*}\leq C_L \|(v,\sigma)-(w,\rho)\|_{*},
\end{align}
where $C_L=\min\left\{  \frac{C C_K^2}{h^{2}\sqrt{2S\lambda^{-1} C_0 (1-\alpha)}}, 1\right\}$ such that $S\coloneqq\min \{\left(1-2\lambda \|\nabla u\|_{\infty} - 2\lambda \right),4\alpha(1-\alpha) \}$ and $ C_0 = \min \left\{   	\frac{1- 2\lambda  \|\nabla u\|_{\infty}}{2\alpha},1-\alpha \right\}$. 
\end{lemma}
\begin{remark}
The Lipschitz constant $C_L$ is global, and the inverse scaling with $h$ will prevent $C_L$ from being small.  However, following Lemma \ref{lem:stability}, local contractiveness is possible if $v$ and $w$ are sufficiently close to the true solution velocity, and if the problem data is sufficiently small.
\end{remark}

\begin{proof}
Let's subtract \eqref{discretedivfreeOldroydB} with $(w,\rho)\in (V_h,\Sigma_h)$ from \eqref{discretedivfreeOldroydB} with $(u,\sigma) \in (V_h,\Sigma_h)$ such that $G(w,\rho)=\left(G_{1}(w,\rho),G_2(w,\rho)\right)=(\tilde w,\tilde\rho)$ and $G(u,\sigma)=\left(G_{1}(u,\sigma),G_2(u,\sigma)\right)=(\tilde u,\tilde\sigma)$, respectively. Then, denoting $e_1=u-w$, $e_1^G=G_1(u,\sigma)-G_1(w,\rho)$ and $e_2=\sigma-\rho$ $e_2^G=G_2(u,\sigma)-G_2(w,\rho)$ provides
\begin{equation}
\begin{aligned}
	\left(e_2^G,\tau\right) 
	+
	\lambda  ( 
	b(e_1,G_{2}(u,\sigma),\tau) 
	+
	b(w,e_2^G,\tau) 
	- \left(\nabla e_1 G_{2}(w,\rho),\tau\right)
	\\
	- \left(\nabla u e_2^G,\tau\right)
	-\left(G_{2}(w,\rho) (\nabla e_1)^T,\tau\right) 
	-\left(e_2^G\nabla u^T,\tau\right) )
	-
	2\alpha  \left(\mathbb{D}(e_1^G),\tau\right) 
	=
	0,
	\\
	\left(e_2^G,\mathbb{D}(v)\right)
	+
	2(1-\alpha)\left(\mathbb{D}(e_1^G),\mathbb{D}(v)\right) 
	=
	0.
\end{aligned}
\end{equation}
Multiply the second equation by $2\alpha$ and add to the first equation, then choose $\tau=e_1^G$ and $v=e_2^G$ to get
\begin{align*}
\|e_2^G\|^2
+
4\alpha(1-\alpha)\|\mathbb{D}(e_1^G)\|^2 
=&
-
\lambda b(e_1,G_{2}(u,\sigma),e_2^G) 
+
\lambda \left(\nabla e_1 G_2(w,\rho),e_2^G\right)
\\
&+
\lambda \left(\nabla u e_2^G,e_2^G\right)
+
\lambda\left(G_{2}(w,\rho) (\nabla e_1)^T,e_2^G\right) 
+
\lambda\left(e_2^G\nabla u^T,e_2^G\right),
\end{align*}
where $b(w,e_2^G,e_2^G) 
$ vanishes. Applying H\"older's with Sobolev embedding now gives
\begin{align*}
\|e_2^G\|^2
+
4\alpha(1-\alpha)\|\mathbb{D}(e_1^G)\|^2 
\leq&
\lambda C \|\nabla e_1\| \|\nabla G_{2}(u,\sigma)\| \|\nabla e_2^G\|
+
\lambda \|\nabla e_1\|  \|\nabla G_{2}(w,\rho)\| \|\nabla e_2^G\|
\\
&+
2 \lambda \|\nabla u\|_{\infty} \|e_2^G\|^2 
+
\lambda \|\nabla G_{2}(w,\rho)\| \|(\nabla e_1)^T\| \|\nabla e_2^G\| .
\end{align*}
Now, by  \eqref{korn} and inverse estimate \eqref{invineq}, we obtain
\begin{align*}
(1-2\lambda \|\nabla u\|_{\infty})\|e_2^G\|^2
+
4\alpha(1-\alpha)\|\mathbb{D}(e_1^G)\|^2 
\leq&
\lambda C C_K h^{-2}  \frac{C_K}{\sqrt{4C_0 (1-\alpha)}}\|f\|_{-1} \|\mathbb{D} (e_1)\| \| e_2^G\|,
\end{align*}
where $\| G_{2}(w,\rho)\|$ is bounded by Lemma \eqref{lemma:welldefiniteG}.
Using Young's inequality finally gives
\begin{align*}
(1-2\lambda \|\nabla u\|_{\infty} -2\lambda)\|e_2^G\|^2
+
4\alpha(1-\alpha)\|\mathbb{D}(e_1^G)\|^2 
\leq&
\lambda C^2 h^{-4}  \frac{C_K^4}{2 C_0 (1-\alpha)}\|f\|^2_{-1} \|\mathbb{D} (e_1)\|^2.
\end{align*}
Dividing both sides by  $S\coloneqq\min \{\left(1-2\lambda \|\nabla u\|_{\infty} - 2\lambda \right),4\alpha(1-\alpha) \}$ and taking the square root of both sides provides the Lipschits continuity bound \eqref{eqn:lipschitz} on $*-$norm with the fact that $\| \mathbb{D}(e_1) \|^2 \leq \| (e_1,e_2) \|_*^2 :=  \| \mathbb{D}(e_1) \|^2 + \| e_2\|^2.$
\end{proof}
Next, we establish that $G$ is Fr\'echet differentiable with Lipschitz continuous derivative. To do that, we first define the operator $G'$, and then demostrate it serves as the Fr\'echet derivative operator of $G$.

\begin{definition}\label{def:G'}
For each \((u,\sigma)\in (V_h, \Sigma_h)\), consider the operator
\[
G'(u,\sigma;\,\cdot,\cdot)\colon (V_h, \Sigma_h)\;\to\;(V_h, \Sigma_h),
\]
defined by
\[
G'(u,\sigma;\,t,s)\;\coloneqq\;\bigl(G_1'(u,\sigma;\,t,s),\;G_2'(u,\sigma;\,t,s)\bigr),
\]
which satisfies
\begin{equation}\label{eqn:G'}
\begin{aligned}
\left(G'_{2}(u,\sigma;t,s),\tau\right) 
+
\lambda  (
b(u,G'_{2}(u,\sigma;t,s),\tau)
+b(t,G_{2}(u,\sigma),\tau)  
- \left(\nabla u G'_{2}(u,\sigma;t,s),\tau\right)&
\\
- \left(\nabla t G_{2}(u,\sigma),\tau\right)
-\left(G'_{2}(u,\sigma;t,s)\nabla u^T,\tau\right) 
-\left(G_{2}(u,\sigma)\nabla t^T,\tau\right) 
)
-
2\alpha  \left(\mathbb{D}(G'_{1}(u,\sigma;t,s)),\tau\right) 
&=
0,
\\
\left(G'_{2}(u,\sigma;t,s),\mathbb{D}(v)\right)
+
2(1-\alpha)\left(\mathbb{D}(G'_{1}(u,\sigma;t,s)),\mathbb{D}(v)\right) 
&=
0,
\end{aligned}
\end{equation}
for all \((t,s)\in (V_h, \Sigma_h)\).
\end{definition}

\begin{lemma}\label{lemma:well-def-G'}
The operator $G'$ introduced in Definition \ref{def:G'} is well-defined for all  $u,t\in V_h $ and $\sigma,s\in\Sigma_h$ such that
\[
\|G'(u,\sigma;t,s)\|_{*} \leq C_G \|(t,s)\|_{*}
\]
where $C_G=C_L$.
\end{lemma}
\begin{proof}
Setting $\tau=G'_{2}(u,\sigma;t,s)$	 and $v=G'_{1}(u,\sigma;t,s)$, then multiplying second equation by $2\alpha$ in \eqref{eqn:G'} and adding it to first one produces 
\begin{align*}
\| G'_{2}(u,\sigma;t,s)\|^2
+&
4\alpha(1-\alpha)\|\mathbb{D}(G'_{1}(u,\sigma;t,s))\|^2
=
-
\lambda b(t,G_{2}(u,\sigma),G'_{2}(u,\sigma;t,s)) 
\\
&+
\lambda \left(\nabla u G'_{2}(u,\sigma;t,s),G'_{2}(u,\sigma;t,s) \right)
+ 
\lambda \left(\nabla t G_{2}(u,\sigma),G'_{2}(u,\sigma;t,s)\right)
\\
&+
\lambda \left(G'_{2}(u,\sigma;t,s)\nabla u^T,G'_{2}(u,\sigma;t,s)\right) 
+
\lambda \left(G_{2}(u,\sigma)\nabla t^T,G'_{2}(u,\sigma;t,s)\right) .
\end{align*}
 Applying H\"older's with Sobolev embedding gives,
\begin{align*}
\| G'_{2}(u,\sigma;t,s)\|^2
+
4\alpha(1-\alpha)\|\mathbb{D}(G'_{1}(u,\sigma;t,s))\|^2
\leq&
\lambda C C_K \|\mathbb{D} (t)\| \|\nabla G_{2}(u,\sigma)\| \|\nabla G'_{2}(u,\sigma;t,s)\| 
\\&
+
2 \lambda \|\nabla u\|_{\infty} \|\nabla G'_{2}(u,\sigma;t,s)\|^2.
\end{align*}
Then, using  \eqref{korn} and inverse estimate \eqref{invineq} along with Lemma \eqref{lemma:welldefiniteG} and
Young's inequality gives
\begin{align*}
(1-2\lambda \|\nabla u\|_{\infty} -2\lambda)\| G'_{2}(u,\sigma;t,s)\|^2
+
4\alpha(1-\alpha)\|\mathbb{D}(G'_{1}(u,\sigma;t,s))\|^2
\leq&
\lambda C^2 h^{-4} \frac{C_K^4}{2 C_0 (1-\alpha)} \|f\|_{-1}^2 \|\mathbb{D} (t)\|^2.
\end{align*}
Dividing both sides by  $S\coloneqq\min \{\left(1-2\lambda \|\nabla u\|_{\infty} - 2\lambda \right),4\alpha(1-\alpha) \}$ and taking the square root of both sides conclude the proof.
\end{proof}

Next, we confirm that  $G'$ operator as the Fr\'echet derivative operator of $G$. Specifically, for each $(u,\sigma)\in(V_h,\Sigma_h)$, there
exists a constant $\mathcal{F}$ such that 
\[
\|G(u+t,\sigma+s)-G(u,\sigma)-G'(u,\sigma;t,s)\|_{*}\leq \mathcal{F} \|(t,s)\|_{*}^2,
\]
for any  $(t,s)\in(V_h,\Sigma_h)$.
\begin{lemma}\label{lemma:frechet}
For arbitrary $(u,\sigma)\in(V_h,\Sigma_h)$ and sufficiently small $(t,s)\in(V_h,\Sigma_h)$, the bound
\begin{align}\label{eqn:frechet}
\|G(u+t,\sigma+s)-G(u,\sigma)-G'(u,\sigma;t,s)\|_{*}
\leq &
\mathcal{F}\| (t,s)\|_{*}^2
\end{align}
holds, with $\mathcal{F}=\frac{ C C_L C_K}{h^{2}\sqrt{8S\lambda^{-1}}}$.
\end{lemma}
\begin{proof}


Subtracting sum of \eqref{discretedivfreeOldroydB} with  $(u,\sigma) \in (V_h,\Sigma_h)$ and \eqref{eqn:G'} from \eqref{discretedivfreeOldroydB} with $(u+t,\sigma+s)\in(V_h,\Sigma_h)$ produce
\begin{align*}
\left(\eta_2,\tau\right) 
+
\lambda  ( 
b(u,\eta_2,\tau) 
+b(t,G_{2}(u+t,\sigma+s)-G_{2}(u,\sigma),\tau) 
- \left(\nabla u\eta_2,\tau\right)
- \left(\nabla t G_{2}(u,\sigma),\tau\right)
\\
-\left(\eta_2\nabla u^T,\tau\right) 
-\left(G_{2}(u,\sigma)\nabla t^T,\tau\right)
)
-
2\alpha  \left(\mathbb{D}(\eta_1),\tau\right) 
&=
0,
\\
\left(\eta_2,\mathbb{D}(v)\right)
+
2(1-\alpha)\left(\mathbb{D}(\eta_1),\mathbb{D}(v)\right) 
&=
0.
\end{align*}
where $\eta_1=G_1(u+t,\sigma+s)-G_1(u,\sigma)-G'_1(u,\sigma;t,s)$ and $\eta_2=G_2(u+t,\sigma+s)-G_2(u,\sigma)-G'_2(u,\sigma;t,s)$.

Then, we choose $\tau=\eta_1$ and $v=\eta_2$ which vanishes the first nonlinear term, multiply second equation by $2\alpha$ and add to the first one to get
\begin{align*}
\|\eta_2\|^2
+
4\alpha(1-\alpha)\|\mathbb{D}(\eta_1)\|^2 
= &
-\lambda b(t,G_{2}(u+t,\sigma+s)-G_{2}(u,\sigma),\eta_2) 
+\lambda \left(\nabla u\eta_2,\eta_2\right)
\\&+\lambda \left(\nabla t(G_{2}(u+t,\sigma+s)-G_{2}(u,\sigma)),\eta_2\right)
+\lambda\left(\eta_2\nabla u^T,\eta_2\right) 
\\&
+\lambda\left((G_{2}(u+t,\sigma+s)-G_{2}(u,\sigma))\nabla t^T,\eta_2\right).
\end{align*}
The H\"older and Sobolev inequalities provide us with
\begin{align*}
\|\eta_2\|^2
+
4\alpha(1-\alpha)\|\mathbb{D}(\eta_1)\|^2 
\leq &
\lambda C\|\nabla t\| \|\nabla (G_{2}(u+t,\sigma+s)-G_{2}(u,\sigma))\| \|\nabla \eta_2\|
+2 \lambda \|\nabla u\|_{\infty}  \|\eta_2\|^2.
\end{align*}
Using an inverse estimate and \eqref{korn}, followed by Young's inequality, we get that
\begin{align*}
(1-2\lambda \|\nabla u\|_{\infty}-2\lambda)\|\eta_2\|^2
+
4\alpha(1-\alpha)\|\mathbb{D}(\eta_1)\|^2 
\leq &
\frac{\lambda}{8} C^2 h^{-4} C_L^2 C_K^2\|\mathbb{D} (t)\|^2 \| (t,s)\|_{*}^2,
\end{align*}
thanks to Lemma \eqref{lemma:lipschitz}. Then, using the fact that $\|\mathbb{D}(t)\|^2 \leq \|\mathbb{D}(t)\|^2+\|s\|^2=\|(t,s)\|^2_{*}$, we conclude \eqref{eqn:frechet}.
\end{proof}

The final step is to prove that the Lipschitz continuity of $G'$ over $V_h\times\Sigma_h$.

\begin{lemma}\label{lemma:lipschitz-G'}
The operator $G$ is Lipschitz continuously differentiable on  $V_h\times\Sigma_h$ where its derivative denoted by $G'$, such that
\begin{align}\label{eqn:lipschitz-G'}
\|G'(u+t,\sigma+s;\phi,\xi)-G'(u,\sigma;\phi,\xi)\|_{*}
\leq&
\hat C_L \|(\phi,\xi)\|_{*}  \| (t,s)\|_{*},
\end{align}
where $\hat C_L = \frac{ C_L C_K C }{2h^2 \sqrt{S \lambda^{-1}}}$, for all $(u,\sigma),(t,s),(k,l)\in(V_h,\Sigma_h)$.
\end{lemma}
\begin{proof}
Let's subtract \eqref{eqn:G'} with $G'(u,\sigma;\phi,\xi)$ from \eqref{eqn:G'} with $G'(u+t,\sigma+s;\phi,\xi)$ and denoting $e_1=G'_{1}(u+t,\sigma+s;\phi,\xi)-G'_{1}(u,\sigma;\phi,\xi)$ and $e_2=G'_{2}(u+t,\sigma+s;\phi,\xi)-G'_{2}(u,\sigma;\phi,\xi)$ to get
\begin{align*}
\left(e_2,\tau\right) 
+
\lambda b(u,e_2,\tau)
+
\lambda b(t,G'_{2}(u+t,\sigma+s;\phi,\xi),\tau)
+
\lambda b(\phi,G_{2}(u+t,\sigma+s)-G_{2}(u,\sigma),\tau)  
\\- \lambda \left(\nabla (u) e_2,\tau\right)
- \lambda \left(\nabla (t) G'_{2}(u+t,\sigma+s;\phi,\xi),\tau\right)
- \lambda \left(\nabla \phi (G_{2}(u+t,\sigma+s)-G_{2}(u,\sigma)),\tau\right)
\\- \lambda \left(e_2\nabla (u)^T,\tau\right) 
- \lambda \left(G'_{2}(u+t,\sigma+s;\phi,\xi)\nabla (t)^T,\tau\right) 
-\lambda\left((G_{2}(u+t,\sigma+s)-G_{2}(u,\sigma))\nabla \phi^T,\tau\right) 
\\- 2\alpha  \left(\mathbb{D}(e_1),\tau\right)
&=0,
\\
\left(e_2,\mathbb{D}(v)\right)
+
2(1-\alpha)\left(\mathbb{D}(e_1),\mathbb{D}(v)\right) 
&=
0.
\end{align*}
By setting $\tau=e_2$ and $v=e_1$, the second term on the left hand side of the first equation vanishes. Then, multiplying the second one by $2\alpha$ ad adding two equations together yields 
\begin{align*}
\|e_2\|^2
+4\alpha(1-\alpha)\|\mathbb{D}(e_1)\|^2
+
\lambda b(t,G'_{2}(u+t,\sigma+s;\phi,\xi),e_2)
+
\lambda b(\phi,G_{2}(u+t,\sigma+s)-G_{2}(u,\sigma),e_2)  
\\- \lambda \left(\nabla (u) e_2,e_2\right)
- \lambda \left(\nabla (t) G'_{2}(u+t,\sigma+s;\phi,\xi),e_2\right)
- \lambda \left(\nabla \phi (G_{2}(u+t,\sigma+s)-G_{2}(u,\sigma)),e_2\right)
\\- \lambda \left(e_2\nabla (u)^T,e_2\right) 
- \lambda \left(G'_{2}(u+t,\sigma+s;\phi,\xi)\nabla (t)^T,e_2\right) 
-\lambda\left((G_{2}(u+t,\sigma+s)-G_{2}(u,\sigma))\nabla \phi^T,e_2\right) 
&=0.
\end{align*}
Using H\"older's inequality with Sobolev embedding, we get
\begin{align*}
(1-2 \lambda \|\nabla u\|_{\infty} )\|e_2\|^2
+4\alpha(1-\alpha)\|\mathbb{D}(e_1)\|^2
\leq&
\lambda C_1 \|\nabla t\| \|\nabla G'_{2}(u+t,\sigma+s;\phi,\xi)\| \|\nabla e_2\|
\\&
+ \lambda C_2 \|\nabla \phi\| \|\nabla (G_{2}(u+t,\sigma+s)-G_{2}(u,\sigma))\| \|\nabla e_2\|.
\end{align*}

Applying the inverse estimate \eqref{invineq} and \eqref{korn}, along with Young's inequality and the result from Lemmas \ref{lemma:lipschitz} and \ref{lemma:well-def-G'}, we derive
\begin{multline*}
(1-2\lambda\|\nabla u\|_{\infty}-2\lambda)\|e_2\|^2
+4\alpha(1-\alpha)\|\mathbb{D}(e_1)\|^2
\\
\leq
\frac{\lambda}{4} C_G^2 C_1^2 C_K^2 h^{-4} \|\mathbb{D} (t)\|^2 \|(\phi,\xi)\|_{*}^2
+ \frac{\lambda}{4} C_2^2 h^{-4} C_K^2\|\mathbb{D}( \phi)\|^2 \| (t,s)\|_{*}^2.
\end{multline*}
By the fact that $C_G=C_L$ and $\|\mathbb{D}  (t)\|^2\leq \|(t,s)\|_*^2$ and $\|\mathbb{D} ( \phi)\|^2\leq \|(\phi,\xi)\|_*^2$, , dividing both sides by  $S\coloneqq\min \{\left(1-2\lambda \|\nabla u\|_{\infty} - 2\lambda \right),4\alpha(1-\alpha) \}$ and  setting $C=\max\{C_1,C_2\}$, we get
\begin{align*}
\|(e_1,e_2)\|^2_{*}
\leq&
\frac{ C_L^2 C_K^2 C^2 }{4 S h^{4}\lambda^{-1}}\|(\phi,\xi)\|_{*}^2  \| (t,s)\|_{*}^2.
\end{align*}
Finally, by taking the square root on both sides, we confirm that  $G'$ is Lipschitz continuous, and \eqref{eqn:lipschitz-G'} is satisfied.
\end{proof}

Now, under the above analysis, we have established the accelerated convergence of \eqref{divfreeOldroydBPicardScheme}.

\begin{cor}
Theorem \ref{thm:genm} applies to the fixed-point operator \( g = G(u,\sigma) \) from Definition \ref{defn:G}, which corresponds to the iteration scheme \eqref{divfreeOldroydBPicardScheme} where the constants in the theorem are given by \( C_0 = C_L \) from Lemma \ref{lemma:lipschitz} and \( C_1 = \hat{C}_L \) from Lemma \ref{lemma:lipschitz-G'}.
\end{cor}
    
\section{Numerical experiments}

In order to investigate the proposed AA-Picard scheme for the viscoelastic Oldroyd-B model, two well-known benchmark experiments were carried out: flow of Oldroyd-B fluid past a circular cylinder, and flow of Oldroyd-B fluid in an L-shaped domain resulting from a contraction geometry with ratio 4:1. Both experiments were implemented in Python through the open source finite element FENICS library \cite{logg2012fenics} following the implementation presented in \cite{TUNC2023}.  As expected in this problem setting, the linear solve time associated with the Picard iteration dominates the run time, and the AA step is nearly negligible.

\subsection{Flow past a cylinder experiment}

In this first experiment we simulate a two-dimensional Oldroyd-B fluid flow past a circular cylinder between two parallel horizontal plates.  This is a standard benchmark test for viscoelastic flows, and we follow the setup presented in \cite{behr2004stabilized}. The modeled domain assumes a reflective symmetry of the flow with respect to the central horizontal line equidistant from the two plates (line $y=0$ in figure \ref{fig:domain}). Thus the model domain is the rectangle $[-10,20]\times [0,8]$ where the cylinder has a unit radius ($R=1$) with its center at the origin of the Cartesian coordinate system ($C=(0,0)$) giving a ratio of cylinder diameter to channel width equal to $\frac18$. No slip boundary conditions are prescribed on the cylinder surface and on the top plate. A fully developed parabolic flow is prescribed on the upstream $\{-10\}\times [0,8]$ and downstream $\{20\}\times[0,8]$ boundaries by

\begin{align}
    u & = \left(
    \tfrac{3}{2}\left[1-\frac{y^2}{h^2}\right] ,
    0
    \right), \\
    T & = \left[
    \begin{matrix}
        2 \lambda \mu_1 \left(-3 \frac{y}{h^2} \right)^2
        & \mu_1 \left(-3 \frac{y}{h^2} \right) \\
        \mu_1 \left(-3 \frac{y}{h^2} \right)
        & 0
    \end{matrix}
    \right].
\end{align}
Finally, a symmetry boundary condition is used downstream and upstream of the cylinder in the line $y=0$ by setting the vertical velocity component to zero. The polymer viscosity $\mu_1=\alpha$ was set to $0.41$ and consequently the solvent viscosity $\mu_2=1-\alpha$ was set to $0.59$.

\begin{figure}
    \centering
    \includegraphics[width=0.8\linewidth]{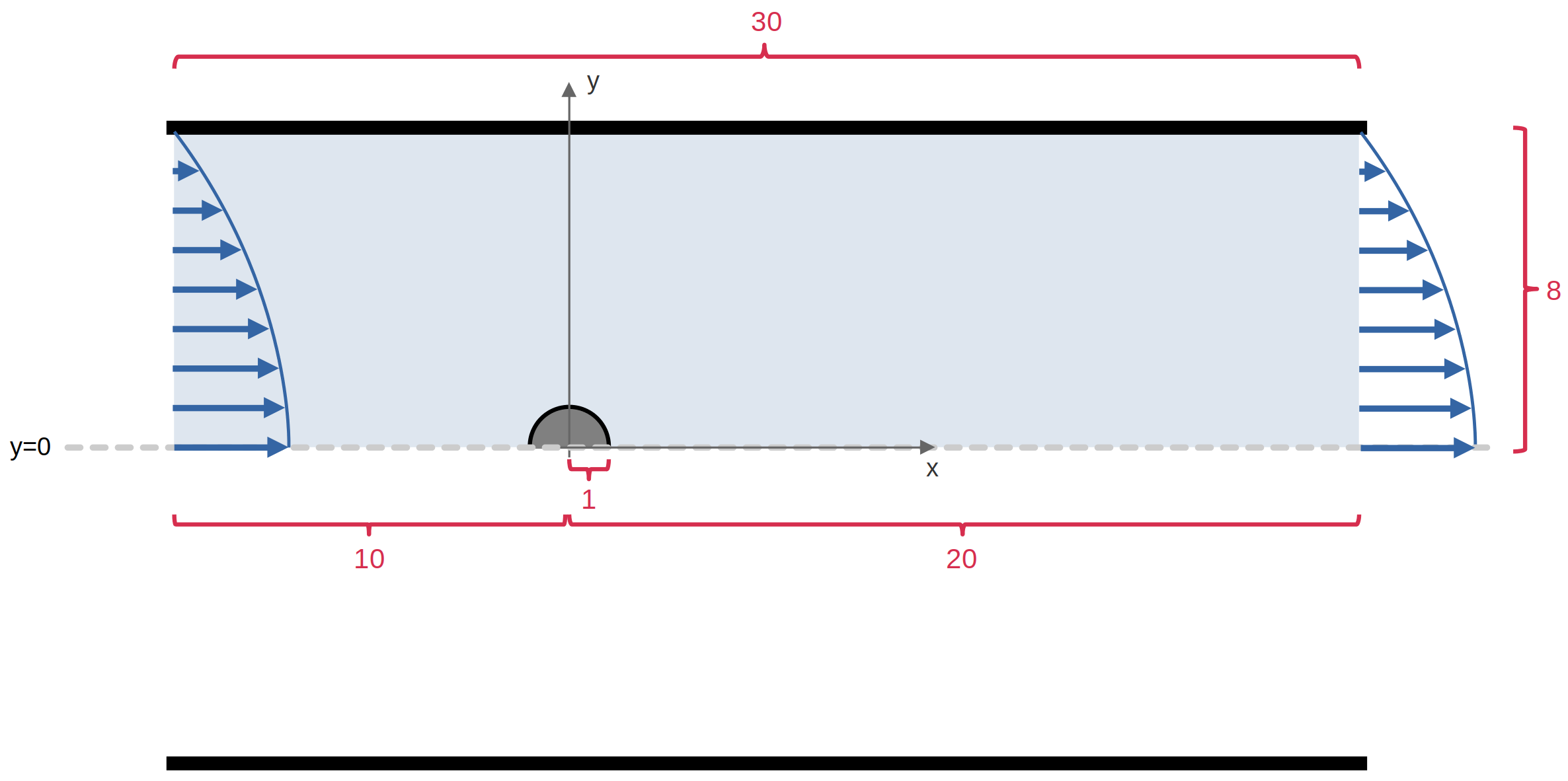}
    \caption{Flow of an Oldroyd-B past a circular cylinder domain.}
    \label{fig:domain}
\end{figure}

The domain was discretized through a mesh formed by triangular elements as follows: starting from an initial coarse resolution mesh, first the region of the rectangle $[-5,9] \times [0,6]$ around the cylinder was refined. Then, in this refined region, a second refinement was carried out on the circle of radius 3 concentric to the cylinder.  The resulting mesh, presented in figure \ref{fig:mesh}, has 29,339 triangular elements.

\begin{figure}
    \centering
    \includegraphics[width=1.0\linewidth]{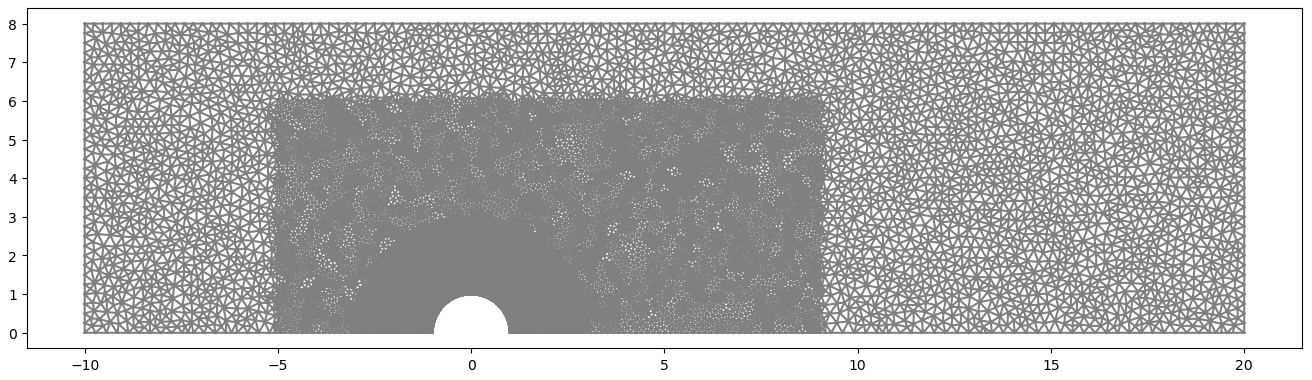}
    \caption{Flow of an Oldroyd-B past a circular cylinder mesh.}
    \label{fig:mesh}
\end{figure}

In this experiment we compare the Picard iteration scheme (i.e. (\ref{divfreeOldroydBPicardScheme}) using $m=0$ and $\beta=1.0$) with AA-Picard using varying $m$ and $\beta$.  Simulations were performed using AA-Picard with three parameter settings: (a) with depth $m=1$ and damping parameter $\beta_k=1.0$, (b) $m=10$ and $\beta_k=0.7$ and (c) $m=20$ and $\beta_k=0.5$.
In both the Picard and AA-Picard methods, the approximate solution was deemed converged when the nonlinear residual fell below $10^{-6}$ in the infinity norm. Table \ref{table:ExpCylinder} displays the number of iterations required for convergence with Picard and AA-Picard across different Weissenberg numbers. Although not presented in Table \ref{table:ExpCylinder}, we emphasize that the drag force exerted on cylinder calculated with the solutions of both iterative schemes are in agreement with the drag force on cylinder obtained in \cite{SunEtAl1999} and \cite{behr2004stabilized} for all values of the Weissenberg number tested. 

In Table \ref{table:ExpCylinder} it is seen that, when comparing the two iterations tested, the Picard iteration diverges for the Weissenberg number greater than or equal to $1.2$ while the AA-Picard iteration produced convergent solutions up to the Weissenberg number up to $1.6$ (for $m=20$ and $\beta=0.5$ and $m=10$ and $\beta=0.7$ cases).
Thus, we see that the use of AA allowed the extension of the values of the Weissenberg number to which the Picard iteration converges. 
Furthermore, Table \ref{table:ExpCylinder} demonstrates that incorporating AA into the Picard iteration significantly decreased the number of iterations required to compute the numerical solution of the model.
In Table \ref{table:ExpCylinder}, we also observe that the effectiveness of AA is sensitive to the choice of parameters m and beta.
We observe that increasing the value of m resulted in reducing the number of iterations although we note that there is a certain upper limit for m beyond which there is no further effect on reducing the number of iterations.
We also observed that reducing the value of the beta parameter also helps both in convergence and in reducing the number of iterations, however for beta we observed the existence of an optimal value (around 0.5) from which reducing the beta value increases the number of iterations again due to the weak upgrade in the input solution of the scheme.

\begin{table}[]
\centering
\caption{Comparison between the Picard iteration and the AA-Picard iteration in the flow of an Oldroyd-B model past a circular cylinder experiment.}
\begin{tabular}{c|c|c|c|c}
\hline
 &  Picard iteration & \multicolumn{3}{c}{Anderson/Picard iteration}            \\  \cline{2-5} 
 &  & $m=1$ and $\beta=1.0$ & $m=10$ and $\beta=0.7$ & $m=20$ and $\beta=0.5$ \\
$\lambda$ & No. iter.           & No. iter. & No. iter. & No. iter. \\ \hline
0.0     & 2                    & 2                 & 3                 & 3                \\
0.1     & 5                    & 5                 & 7                 & 6              \\
0.2     & 7                    & 7                 & 7                 & 7               \\
0.3     & 8                    & 8                 & 8                 & 7               \\
0.4     & 10                   & 9                 & 9                 & 8               \\
0.5     & 12                   & 11                & 10                & 9              \\
0.6     & 14                   & 13                & 11                & 10              \\
0.7     & 17                   & 14                & 12                & 11             \\
0.8     & 21                   & 17                & 13                & 12              \\
0.9     & 29                   & 23                & 15                & 13              \\
1.0     & 39                   & 24                & 17                & 15              \\
1.1     & 101                  & 34                & 19                & 17              \\
1.2     & \textbf{Diverge}     & 50                & 22                & 18              \\
1.3     & \textbf{Diverge}     & 84                & 24                & 21              \\
1.4     & \textbf{Diverge}     & 125               & 29                & 24              \\
1.5     & \textbf{Diverge}     & \textbf{Diverge}  & 35                & 29              \\
1.6     & \textbf{Diverge}     & \textbf{Diverge}  & 55                & 42              \\
1.7     & \textbf{Diverge}     & \textbf{Diverge}  & \textbf{Diverge}  & \textbf{Diverge}
\end{tabular}\label{table:ExpCylinder}
\end{table}

\subsection{Flow of Oldroyd-B fluid in a contracted L-shaped geometry}

In our second experiment we reproduced the simulation presented in \cite{ervin2008} in order to compare the performance of Picard and AA-Picard. This experiment consists of the flow of an Oldroyd-B fluid in a planar channel with a contraction with a ratio of 4:1 between the upstream and downstream widths (see Figure \ref{fig:domain2}). 
It is known that the flow changes at the upstream channel from a fully developed Poiseuille flow at the inlet end to a flow having a vortex in the central contraction region of the re-entrant corner and then returns at the downstream channel to a fully developed Poiseuille flow at the outlet end.  
We set the polymer viscosity $\mu_1=\alpha=\frac89$, the solvent viscosity $\mu_2=1-\alpha=\frac19$ and the Weissenberg number was set to $\lambda=\frac{7}{10}$.
The meshes tested in this experiment are listed in table \ref{table:ExpContraction} along with its corresponding DoF. 
Figure \ref{fig:41mesh} presents the coarsest L-shaped geometry mesh in our experiment which has 8,631 DoF when $dx=0.25$ and $dy=0.0625$.
The simulation with Anderson acceleration used the damping factors $\beta_k$ set to $0.5$ and the depth $m$ to 20.
In both Picard and AA-Picard iterations the approximated solution was considered achieved when the nonlinear residual between the last two solutions was smaller than $10^{-6}$ in the infinity norm.

Table \ref{table:ExpContraction} presents the results obtained with both pure Picard and AA-Picard iterations along with the values of $\Vert u\Vert_0$, $\Vert u \Vert_1$ and $\Vert s \Vert_0$ presented in \cite{ervin2008}. 
Comparing the Picard and AA-Picard number of iterations for different mesh resolutions we can see that the number of iteration when using the Picard iteration was very sensitive to the mesh resolution. 
Moreover, in the 256x128 case the Picard iteration diverges. 
On the other hand, we see in Table \ref{table:ExpContraction} that the AA-Picard iteration was convergent for all meshes and had the number of iterations much less sensitive to mesh resolution when compared to Picard.
Thus, in this second experiment we could also see that the use of AA along with the Picard iteration allowed both a lower sensitivity of the iterative method to the mesh resolution and also a greater velocity of convergence when compared to the Picard iteration.

\begin{figure}
    \centering
    \includegraphics[width=1.0\linewidth]{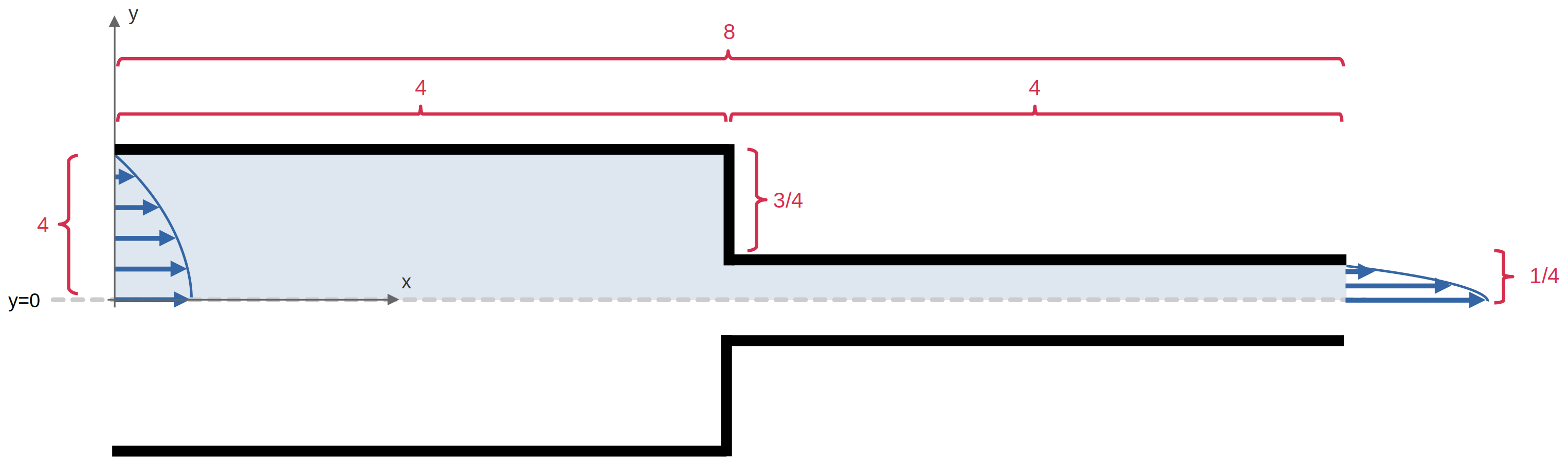}
    \caption{Flow of an Oldroyd-B past a circular cylinder mesh.}
    \label{fig:domain2}
\end{figure}

\begin{figure}
    \centering
    \includegraphics[width=1.0\linewidth]{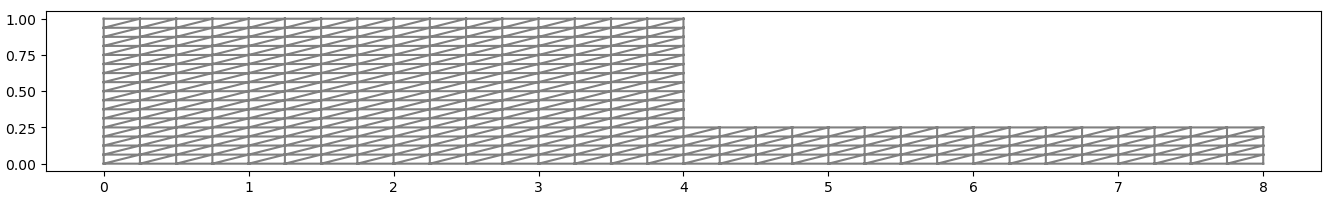}
    \caption{The 32x64 L-shaped geometry mesh.}
    \label{fig:41mesh}
\end{figure}

\begin{table}[]
\centering
\caption{Comparison between the pure Picard Iteration and the Picard Iteration boosted by the Anderson acceleration in the flow of an Oldroyd-B model in a contracted L-shaped geometry experiment.}
\begin{tabular}{|c|c|c|c|c|c|c|c|}
\hline
\multicolumn{2}{|c|}{Mesh} & 256x128  & 32x16    & 64x32    & 128x64   & 256x128  & 512x256  \\ \hline
\multicolumn{2}{|c|}{DoF} & 114811   & 8631     & 33255    & 130503   & 516999   & 2057991  \\ \hline
\multicolumn{2}{|c|}{$\text{dx}$}  & 0.03125 & 0.25     & 
0.125 & 0.0625 & 0.03125   & 0.015625  \\ \hline
\multicolumn{2}{|c|}{$\text{dy}$}  & 0.0078125 & 0.625     & 
0.03125    & 0.015625 & 0.0078125   & 0.00390625  \\ \hline
             & $\Vert u \Vert_0$   & 0.104166 &          &          &          &          &          \\ \cline{2-8} 
Ervin et al. & $\Vert u \Vert_1$ & 0.595209 &          &          &          &          &          \\ \cline{2-8} 
(2008)       & $\Vert s \Vert_0$ & 0.932091 &          &          &          &          &          \\ \cline{2-8} 
             & No. of iter. & (not given)        &          &          &          &          &          \\ \hline
             & $\Vert u \Vert_0$ &          & 0.104152 & 0.104151 & 0.104147 & 0.104144 & -        \\ \cline{2-8} 
Picard       & $\Vert u \Vert_1$ &          & 0.596802 & 0.595947 & 0.595609 & 0.595458 & -        \\ \cline{2-8} 
         & $\Vert s \Vert_0$ &          & 0.929722 & 0.929508 & 0.929386 & 0.929364 & -        \\ \cline{2-8} 
             & No. of iter. &          & 20       & 24       & 38       & 98       & Diverge  \\ \hline
      & $\Vert u \Vert_0$ &          & 0.104152 & 0.104151 & 0.104147 & 0.104144 & 0.104144 \\ \cline{2-8} 
AA-Picard             & $\Vert u \vert_1$ &          & 0.596802 & 0.595947 & 0.595609 & 0.595458 & 0.595398 \\ \cline{2-8} 
     & $\Vert s \Vert_0$ &          & 0.929722 & 0.929508 & 0.929386 & 0.929364 & 0.929383 \\ \cline{2-8} 
 & No. of iter. &          & 11       & 13       & 14       & 16       & 20       \\ \hline
\end{tabular}\label{table:ExpContraction}
\end{table}

\section{Conclusions}

In this study we have introduced, analyzed and evaluated an Anderson accelerated Picard iteration scheme for the Oldroyd-B model of incompressible viscoelastic flows. Firstly we provide a convergence analysis of the proposed nonlinear iterative solver.  Also two numerical simulations, namely the flow around a circular cylinder and the experiment involving an L-shaped geometry, showed that with a good choice of the depth $m$ and damping factor $\beta$, the AA-Picard scheme allows the convergence of the iterative scheme at significantly higher Weissenberg numbers than Picard.  Furthermore, it was observed that the AA-Picard scheme converges within substentially less number iterations when compared to the Picard iteration, in simulations in which both schemes are convergent.

\section{Appendix}

Below we present the main part of our Python/Fenics implementation of the Anderson acceleration algorithm for the Picard iteration solver of the Oldroyd-B model. In general terms we use a peculiar list (\texttt{MyList class}) which we implemented in order to fit the algorithm needs. \texttt{l} is the Weissenberg number ($\lambda$) and \texttt{mu1} and \texttt{mu2} are the polymer ($\mu_1$) and solvent ($\mu_2$) viscosities, respectively. Then, the \texttt{AndersonIteration} function implements the Anderson iteration algorithm. In the \texttt{AndersonIteration} function, \texttt{g} is the Oldroyd-B/Picard solver. \texttt{beta} is the damping factor $\beta_k$ and \texttt{m} is the depth number. At the end of each new Anderson/Picard iteration the \texttt{CalculateError} function calculates the nonlinear residual between the last two solutions.

{\tiny
\begin{lstlisting}[language=Python]
class MyList(list):
    def update(self, index, u, savememory):
        if savememory:
            self.pop(0)
            self.insert(1, x.copy(deepcopy=True))
            self[1].vector()[:] = u
        else:
            self.insert(index, x.copy(deepcopy=True))
            self[index].vector()[:] = u
    def mk_insert(self, index, mk, item):
        N = x.vector()[:].size
        y = x.copy(deepcopy=True)
        y.assign(item)
        if mk-1 == 0:
            self.insert(index, y.vector()[:].reshape(N, 1) )
        else:
            self.pop(0)
            self.insert(1, np.c_[y.vector()[:].reshape(N, 1), self[-1][:, :(mk-1)]] )

def AndersonIteration(domain, oldroydb, anderson):
    global x

    def g(uold, Told):
        w = OBModel.SolveOldroydB(Z, bcs, uold, Told, l, mu1, mu2)
        return w

    x = Function(Z)

    k = 0
    U = MyList([x])
    W = MyList(['none'])
    W.update(1, g(U[0].split()[1], U[0].split()[2]).vector()[:] - U[0].vector()[:], False)
    U.update(1, U[0].vector()[:] + beta * W[1].vector()[:], False)
    error = CalculateError(U)

    E = MyList(['none'])
    F = MyList(['none'])
    gamma = ['none']
    while error['linf'] >= tol and k+1 < kmax:
        k += 1
        W.update(k+1, g(U[-1].split()[1], U[-1].split()[2]).vector()[:] - U[-1].vector()[:], savememory)
        if m == 0:
            U.update(k+1, U[-1].vector()[:] + beta * W[-1].vector()[:], savememory)
        else:
            mk = min(k, m)
            # Creating and updating E and F matrix
            F.mk_insert(k, mk, W[-1]-W[-2])
            E.mk_insert(k, mk, U[-1]-U[-2])
            # Obtaining gamma
            Q, R = np.linalg.qr(F[-1])
            gamma = np.linalg.lstsq(R + (alpha**2)*np.identity(mk), np.dot(Q.T, W[-1].vector()[:]))[0]
            U.update(k+1, U[-1].vector()[:]+beta*W[-1].vector()[:]-np.matmul((E[-1]+beta*F[-1]), gamma), savememory)
        error = CalculateError(U)
    return U[-1]
\end{lstlisting}
}

\bibliographystyle{plain}
\bibliography{graddiv2724}

\begin{thebibliography}{10}

\bibitem{AJW17}
H.~An, X.~Jia, and H.F. Walker.
\newblock Anderson acceleration and application to the three-temperature energy equations.
\newblock {\em Journal of Computational Physics}, 347:1--19, 2017.

\bibitem{Anderson65}
D.~G. Anderson.
\newblock Iterative procedures for nonlinear integral equations.
\newblock {\em Journal of the Association for Computing Machinery}, 12(4):547--560, 1965.

\bibitem{Apelian1988_viscoelastic}
M.~R. Apelian, R.~C. Armstrong, and R.~A. Brown.
\newblock Impact of the constitutive equation and singularity on the calculation of stick-slip flow: The modified upper-convected {M}axwell model ({MUCM}).
\newblock {\em Journal of Non-newtonian Fluid Mechanics}, 27:299--321, 1988.

\bibitem{behr2004stabilized}
M.~Behr, D.~Arora, O.~Coronado-Matutti, and M~Pasquali.
\newblock Stabilized finite element methods of {GLS} type for {Oldroyd-B} viscoelastic fluid.
\newblock In {\em Proceedings of ECCOMAS-4th European Congress on Computational Methods in Applied Sciences and Engineering)}, 2004.

\bibitem{BS08}
S.C. Brenner and L.~R. Scott.
\newblock {\em The mathematical theory of finite element methods}, volume~15 of {\em Texts in Applied Mathematics}.
\newblock Springer, New York, third edition, 2008.

\bibitem{Brown_etal93_viscoelastic}
R.~A. {Brown}, M.~J. {Szady}, P.~J. {Northey}, and R.~C. {Armstrong}.
\newblock {On the numerical stability of mixed finite-element methods for viscoelastic flows governed by differential constitutive equations}.
\newblock {\em Theoretical and Computational Fluid Dynamics}, 5(2-3):77--106, October 1993.

\bibitem{ervin2008}
V.~J. Ervin, S.~J. Howell, and H.~Lee.
\newblock A two-parameter defect-correction method for computation of steady-state viscoelastic fluid flow.
\newblock {\em Applied Mathematics and Computation}, 196(2):818--834, 2008.

\bibitem{Ervin06_viscoelastic}
V.~J. Ervin and H.~Lee.
\newblock Defect correction method for viscoelastic fluid flows at high {W}eissenberg number.
\newblock {\em Numerical Methods for Partial Differential Equations}, 22(1):145--164, 2006.

\bibitem{Ervin05_viscoelastic}
V.~J. Ervin, H.~K. Lee, and L.~N. Ntasin.
\newblock Analysis of the {O}seen-viscoelastic fluid flow problem.
\newblock {\em Journal of Non-Newtonian Fluid Mechanics}, 127(2):157--168, 2005.

\bibitem{FAC19_AA}
M.~Filippini, P.~Alotto, and A.~Giust.
\newblock Anderson acceleration for electromagnetic nonlinear problems.
\newblock {\em Compel}, 38(5):1493--1506, 2019 2019.

\bibitem{Baaijens98_viscoelastic}
{F.P.T. Baaijens}.
\newblock Mixed finite element methods for viscoelastic flow analysis: a review.
\newblock {\em Journal of Non-Newtonian Fluid Mechanics}, 79(2):361--385, 1998.

\bibitem{FZB20}
A.~Fu, J.~Zhang, and S.~Boyd.
\newblock Anderson accelerated {D}ouglas-{R}achford splitting.
\newblock {\em SIAM Journal on Scientific Computing}, 42(6):A3560--A3583, 2020.

\bibitem{GIESEKUS1982_viscoelastic}
H.~Giesekus.
\newblock A simple constitutive equation for polymer fluids based on the concept of deformation-dependent tensorial mobility.
\newblock {\em Journal of Non-Newtonian Fluid Mechanics}, 11(1):69--109, 1982.

\bibitem{JOHNSON77_viscoelastic}
M.W. Johnson and D.~Segalman.
\newblock A model for viscoelastic fluid behavior which allows non-affine deformation.
\newblock {\em Journal of Non-Newtonian Fluid Mechanics}, 2(3):255--270, 1977.

\bibitem{K18}
C.T. Kelley.
\newblock Numerical methods for nonlinear equations.
\newblock {\em Acta Numerica}, 27:207--287, 2018.

\bibitem{KEUNINGS1986_viscoelastic}
R.~Keunings.
\newblock On the high {W}eissenberg number problem.
\newblock {\em Journal of Non-Newtonian Fluid Mechanics}, 20:209--226, 1986.

\bibitem{Laytonbook}
W.~Layton.
\newblock {\em An {I}ntroduction to the {N}umerical {A}nalysis of {V}iscous {I}ncompressible {F}lows}.
\newblock SIAM, Philadelphia, 2008.

\bibitem{LW16}
J.~Loffeld and C.~Woodward.
\newblock Considerations on the implementation and use of {A}nderson acceleration on distributed memory and {GPU}-based parallel computers.
\newblock {\em Advances in the Mathematical Sciences}, pages 417--436, 2016.

\bibitem{logg2012fenics}
A.~Logg, K.-A. Mardal, and G.~Wells.
\newblock {\em Automated solution of differential equations by the finite element method: The {FEniCS} book}, volume~84.
\newblock Springer Science \& Business Media, 2012.

\bibitem{LWWY12}
P.A. Lott, H.F. Walker, C.S. Woodward, and U.M. Yang.
\newblock An accelerated {P}icard method for nonlinear systems related to variably saturated flow.
\newblock {\em Advances in Water Resources}, 38:92--101, 2012.

\bibitem{Luo98_viscoelastic}
X.L. Luo.
\newblock An incremental difference formulation for viscoelastic flows and high resolution {FEM} solutions at high {W}eissenberg numbers.
\newblock {\em Journal of Non-Newtonian Fluid Mechanics}, 79(1):57--75, 1998.

\bibitem{Oldroyd58_viscoelastic}
J.~G. Oldroyd and G.I. Taylor.
\newblock Non-{N}ewtonian effects in steady motion of some idealized elastico-viscous liquids.
\newblock {\em Proceedings of the Royal Society of London. Series A. Mathematical and Physical Sciences}, 245(1241):278--297, 1958.

\bibitem{PDZGQL2018_AA}
Y.~Peng, B.~Deng, J.~Zhang, F.~Geng, W.~Qin, and L.~Liu.
\newblock Anderson acceleration for geometry optimization and physics simulation.
\newblock {\em ACM Trans. Graph.}, 37(4), 2018.

\bibitem{PWW17}
L.~Perrotti, N.~Walkington, and D.~Wang.
\newblock Numerical approximation of viscoelastic fluids.
\newblock {\em M2AN}, 51:1119--1144, 2017.

\bibitem{PETERA2002_viscoelastic}
J.~Petera.
\newblock A new finite element scheme using the {L}agrangian framework for simulation of viscoelastic fluid flows.
\newblock {\em Journal of Non-Newtonian Fluid Mechanics}, 103(1):1--43, 2002.

\bibitem{PR21}
S.~Pollock and L.~Rebholz.
\newblock Anderson acceleration for contractive and noncontractive operators.
\newblock {\em IMA Journal of Numerical Analysis}, 41(4):2841--2872, 2021.

\bibitem{PR23}
S.~Pollock and L.~Rebholz.
\newblock Filtering for {A}nderson acceleration.
\newblock {\em SIAM Journal on Scientific Computing}, 45(4):A1571--A1590, 2023.

\bibitem{PRV23_AA}
S.~Pollock, L.~Rebholz, and D.~Vargun.
\newblock Anderson acceleration for a regularized {B}ingham model.
\newblock {\em Numerical Methods for Partial Differential Equations}, 39(5):3874--3896, 2023.

\bibitem{PRX19}
S.~Pollock, L.~Rebholz, and M.~Xiao.
\newblock Anderson-accelerated convergence of {P}icard iterations for incompressible {N}avier-{S}tokes equations.
\newblock {\em SIAM Journal on Numerical Analysis}, 57:615-- 637, 2019.

\bibitem{PRX21}
S.~Pollock, L.~Rebholz, and M.~Xiao.
\newblock Acceleration of nonlinear solvers for natural convection problems.
\newblock {\em Journal of Numerical Mathematics}, 29:1--19, 2021.

\bibitem{LVX21}
L.~Rebholz, D.~Vargun, and M.~Xiao.
\newblock {Enabling fast convergence of the iterated penalty Picard iteration with $O(1)$ penalty parameter for incompressible Navier-Stokes via Anderson acceleration}.
\newblock {\em Computer Methods in Applied Mechanics and Engineering}, 387(114178):1--17, 2021.

\bibitem{Renardy85_viscoelastic}
M.~Renardy.
\newblock Existence of slow steady flows of viscoelastic fluids with differential constitutive equations.
\newblock {\em ZAMM - Journal of Applied Mathematics and Mechanics / Zeitschrift für Angewandte Mathematik und Mechanik}, 65(9):449--451, 1985.

\bibitem{SM11}
P.~Stasiak and M.W. Matsen.
\newblock Efficiency of pseudo-spectral algorithms with {A}nderson mixing for the {SCFT} of periodic block-copolymer phases.
\newblock {\em The European Physical Journal E}, 34:110:1--9, 2011.

\bibitem{SunEtAl1999}
J~Sun, M.D. Smith, RC~Armstrong, and RA~Brown.
\newblock Finite element method for viscoelastic flows based on the discrete adaptive viscoelastic stress splitting and the discontinuous {G}alerkin method: {DAVSS}-{G/DG}.
\newblock {\em Journal of Non-Newtonian Fluid Mechanics}, 86(3):281--307, 1999.

\bibitem{TKSHCP15}
A.~Toth, C.T. Kelley, S.~Slattery, S.~Hamilton, K.~Clarno, and R.~Pawlowski.
\newblock Analysis of {A}nderson acceleration on a simplified neutronics/thermal hydraulics system.
\newblock {\em Proceedings of the ANS MC2015 Joint International Conference on Mathematics and Computation (M\&C), Supercomputing in Nuclear Applications (SNA) and the Monte Carlo (MC) Method}, ANS MC2015 CD:1--12, 2015.

\bibitem{TUNC2023}
B.~Tunç, G.J. Rodin, and T.~E. Yankeelov.
\newblock Implementing multiphysics models in {FEniCS}: Viscoelastic flows, poroelasticity, and tumor growth.
\newblock {\em Biomedical Engineering Advances}, 5:100074, 2023.

\bibitem{WaNi11}
H.~F. Walker and P.~Ni.
\newblock Anderson acceleration for fixed-point iterations.
\newblock {\em SIAM Journal on Numerical Analysis}, 49(4):1715--1735, 2011.

\bibitem{WSB21_AA}
D.~Wicht, M.~Schneider, and T.~Bohlke.
\newblock Anderson-accelerated polarization schemes for fast {F}ourier transform-based computational homogenization.
\newblock {\em International Journal for Numerical Methods in Engineering}, 122(9):2287--2311, 2021.

\bibitem{Yang2021_AA}
Y.~Yang.
\newblock Anderson acceleration for seismic inversion.
\newblock {\em Geophysics}, 86(1):R99--R108, 2021.

\end{thebibliography}
\end{document}